\documentclass[11pt,twoside]{amsart}
\usepackage{latexsym,amssymb,amsmath}

\numberwithin{equation}{section}
\def\demo{\noindent{\it Proof. }}
\newtheorem{theorem}{Theorem}[section]
\newtheorem{lemma}[theorem]{Lemma}
\newtheorem{proposition}[theorem]{Proposition}
\newtheorem{corollary}[theorem]{Corollary}
\newtheorem{conjecture}[theorem]{Conjecture}

\theoremstyle{definition}
\newtheorem{definition}[theorem]{Definition} 
\newtheorem{procedure}[theorem]{Procedure} 
\newtheorem{remark}[theorem]{Remark}
\newtheorem{example}[theorem]{Example}

\newcommand{\Z}{\mathbb{Z}}

\begin{document}


\title[Generalized minimum distance]
{Generalized minimum distance functions} 

\thanks{The first author was supported by COFAA-IPN and SNI,
Mexico. The second and third author were supported by SNI, Mexico. 
The fourth author was
supported by a scholarship from CONACyT, Mexico}

\author[M. Gonz\'alez-Sarabia]{Manuel Gonz\'alez-Sarabia} 
\address{M.G. Sarabia. Instituto Polit\'ecnico Nacional, 
UPIITA, Av. IPN No. 2580,
Col. La Laguna Ticom\'an,
Gustavo A. Madero C.P. 07340,
 Ciudad de M\'exico. 
Departamento de Ciencias B\'asicas}
\email{mgonzalezsa@ipn.mx}

\author[J. Mart\'\i nez-Bernal]{Jos\'e Mart\'\i nez-Bernal}
\address{
Departamento de
Matem\'aticas\\
Centro de Investigaci\'on y de Estudios Avanzados del IPN\\
Apartado Postal
14--740 \\
07000 Mexico City, D.F.
}
\email{jmb@math.cinvestav.mx}

\author[R. H. Villarreal]{Rafael H. Villarreal}
\address{
Departamento de
Matem\'aticas\\
Centro de Investigaci\'on y de Estudios
Avanzados del
IPN\\
Apartado Postal
14--740 \\
07000 Mexico City, D.F.
}
\email{vila@math.cinvestav.mx}

\author[C. Vivares]{Carlos E. Vivares}
\address{
Departamento de
Matem\'aticas\\
Centro de Investigaci\'on y de Estudios
Avanzados del
IPN\\
Apartado Postal
14--740 \\
07000 Mexico City, D.F.
}
\email{cevivares@math.cinvestav.mx}

\begin{abstract} 
Using commutative algebra methods we study the generalized minimum distance
function (gmd function) and the corresponding generalized footprint
function of a graded ideal in a 
polynomial ring over a field. The number of solutions that a system
of homogeneous polynomials has in
any given finite set of projective points is 
expressed as the degree of a graded ideal.
If $\mathbb{X}$ is a set of projective
points over a finite field and $I$ is its vanishing
ideal, we show
that the gmd function and the Vasconcelos
function of $I$ are equal to the $r$-th 
generalized Hamming weight of the corresponding Reed-Muller-type 
code $C_\mathbb{X}(d)$ of degree $d$. We show that the generalized footprint function 
of $I$ is a lower bound for the $r$-th generalized Hamming
weight of $C_\mathbb{X}(d)$. Then we present some applications to projective nested
cartesian codes. To give applications of our lower bound to
algebraic coding theory, we show an interesting integer inequality. Then we show
an explicit formula and a combinatorial formula
for the second generalized Hamming weight of an affine cartesian 
code.   
\end{abstract}

\maketitle 

\section{Introduction}\label{intro-section}
Let $S=K[t_1,\ldots,t_s]=\oplus_{d=0}^{\infty} S_d$ be a polynomial ring over
a field $K$ with the standard grading and let $I\neq(0)$ be a graded ideal
of $S$. In this work we extend the scope of
\cite{hilbert-min-dis} by
considering generalized footprint and minimum distance functions.
Given $d,r\in\mathbb{N}_+$, let $\mathcal{F}_{d,r}$ be the set:
$$
\mathcal{F}_{d,r}:=\{\, \{f_1,\ldots,f_r\}\subset S_d\, \vert\,
\overline{f}_1,\ldots,\overline{f}_r\, \mbox{are linearly
independent over }K, (I\colon(f_1,\ldots,f_r))\neq I\}, 
$$
where $\overline{f}=f+I$ is the class of $f$ modulo $I$, and 
$(I\colon(f_1,\ldots,f_r))=\{g\in S \vert gf_i \in I,\,\forall\, i\}$ is referred to as
an ideal quotient or colon ideal.

We denote the {\it degree\/} of $S/I$ by $\deg(S/I)$. The function 
$\delta_I\colon \mathbb{N}_+\times\mathbb{N}_+\rightarrow \mathbb{Z}$ given by 
$$
\delta_I(d,r):=\left\{\begin{array}{ll}\deg(S/I)-\max\{\deg(S/(I,F))\vert\,
F\in\mathcal{F}_{d,r}\}&\mbox{if }\mathcal{F}_{d,r}\neq\emptyset,\\
\deg(S/I)&\mbox{if\ }\mathcal{F}_{d,r}=\emptyset,
\end{array}\right.$$
is called the {\it generalized minimum distance function\/} of
$I$, or simply the {\it gmd function} of $I$. If $r=1$ one obtains
the minimum distance function of $I$ \cite{hilbert-min-dis}.   
To compute $\delta_I(d,r)$ is a
difficult problem. One of our aims is to introduce lower bounds 
for $\delta_I(d,r)$ which 
are easier to compute. 

Fix a monomial order $\prec$ on $S$. Let ${\rm in}_\prec(I)$ be the
initial ideal of $I$ and let $\Delta_\prec(I)$  be the
{\it footprint\/} of $S/I$ consisting of all the {\it standard
monomials\/} of $S/I$ with respect 
to $\prec$. The footprint of $S/I$ is also called the {\it
Gr\"obner \'escalier\/} of $I$. Given integers $d,r\geq 1$, let
$\mathcal{M}_{\prec, d, r}$ be the 
set of all subsets $M$ of $\Delta_\prec(I)_d:=\Delta_\prec(I)\cap S_d$
with $r$ distinct elements such 
that $({\rm in}_\prec(I)\colon(M))\neq {\rm
in}_\prec(I)$. 
The {\it generalized footprint
function\/} of $I$, 
denoted ${\rm fp}_I$, is the function ${\rm fp}_I\colon
\mathbb{N}_+\times\mathbb{N}_+\rightarrow \mathbb{Z}$ given
by 
$$
{\rm fp}_I(d,r):=\left\{\begin{array}{ll}\deg(S/I)-\max\{\deg(S/({\rm
in}_\prec(I),M))\,\vert\,
M\in\mathcal{M}_{\prec, d,r}\}&\mbox{if }\mathcal{M}_{\prec, d,r}\neq\emptyset,\\
\deg(S/I)&\mbox{if }\mathcal{M}_{\prec, d,r}=\emptyset.
\end{array}\right.
$$

If $r=1$ one obtains the footprint function of $I$ that was
studied in \cite{footprint-ci} from a theoretical point of
view (see \cite{hilbert-min-dis,min-dis-ci} for some applications).   
The footprint of vanishing ideals of finite sets of affine points was used in the works 
of Geil \cite{geil} and Carvalho \cite{carvalho} to study
affine Reed-Muller-type codes. Long before these two papers appeared the
footprint was used by Geil in connection with all kinds
of codes (including one-point algebraic geometric codes); see
\cite{geil-2008,geil-hoholdt,geil-pellikaan} and the references therein.   

The definition of $\delta_I(d,r)$ was motivated by the notion of 
generalized Hamming weight of a linear code \cite{helleseth,klove,wei}. For convenience we
recall this notion. Let $K=\mathbb{F}_q$ be a finite field and let $C$ be a $[m,k]$ {\it linear
code} of {\it length} $m$ and {\it dimension} $k$, 
that is, $C$ is a linear subspace of $K^m$ with $k=\dim_K(C)$. Let $1\leq r\leq k$ be an integer.  
Given a subcode $D$ of $C$ (that is, $D$ is a linear subspace of $C$),
the {\it support\/} $\chi(D)$ of $D$ is the set of non-zero positions of $D$, that is,  
$$
\chi(D):=\{i\,\vert\, \exists\, (a_1,\ldots,a_m)\in D,\, a_i\neq 0\}.
$$

The $r$-th {\it generalized Hamming weight\/} of $C$, denoted
$\delta_r(C)$, is the size of the smallest support of an
$r$-dimensional subcode. Generalized Hamming weights
have received a lot of attention; see
\cite{carvalho,ghorpade,geil,GHW2014,Pellikaan,Johnsen,olaya,schaathun-willems,tsfasman,
wei,wei-yang,Yang} and the
references therein. The study of these weights is related to 
trellis coding, $t$--resilient functions, and was motivated by 
some applications from cryptography \cite{wei}. 

The minimum distance of projective Reed-Muller-type codes has been
studied using Gr\"obner bases  and commutative algebra techniques; see
\cite{carvalho,carvalho-lopez-lopez,geil,geil-thomsen,hilbert-min-dis,algcodes}
and the references  
therein. In this work we extend these techniques to study the $r$-th
generalized Hamming weight of projective 
Reed-Muller-type codes. These linear codes are constructed
as follows. 

Let $K=\mathbb{F}_q$ be a finite field with $q$ elements,
let $\mathbb{P}^{s-1}$ be a projective space over 
$K$, and let $\mathbb{X}$ be a subset of
$\mathbb{P}^{s-1}$. The {\it vanishing ideal\/} of
$\mathbb{X}$, denoted $I(\mathbb{X})$,  is the ideal of $S$ 
generated by the homogeneous polynomials that vanish at all points of
$\mathbb{X}$. The Hilbert function of $S/I(\mathbb{X})$ is denoted by
$H_\mathbb{X}(d)$. We can write
$\mathbb{X}=\{[P_1],\ldots,[P_m]\}\subset\mathbb{P}^{s-1}$ 
with $m=|\mathbb{X}|$. Here we assume that the first non-zero entry of
each $[P_i]$ is $1$. In the special case that $\mathbb{X}$ has the
form $[X\times\{1\}]$ for some $X\subset\mathbb{F}_q^{s-1}$, we
assume that the $s$-th entry of each $[P_i]$ 
is $1$.  

Fix a degree $d\geq 1$. There is a $K$-linear map given by  
\begin{equation*}
{\rm ev}_d\colon S_d\rightarrow K^{m},\ \ \ \ \ 
f\mapsto
\left(f(P_1),\ldots,f(P_m)\right).
\end{equation*}

The image of $S_d$ under ${\rm ev}_d$, denoted by  $C_\mathbb{X}(d)$, is
called a {\it projective Reed-Muller-type code\/} of
degree $d$ on $\mathbb{X}$ \cite{duursma-renteria-tapia,GRT}. The
points in $\mathbb{X}$ are often called evaluation points in the
algebraic coding context. 
The {\it parameters} of the linear
code $C_\mathbb{X}(d)$ are:
\begin{itemize}
\item[(a)] {\it length\/}: $|\mathbb{X}|$,
\item[(b)] {\it dimension\/}: $\dim_K C_\mathbb{X}(d)$,
\item[(c)] $r$-th {\it generalized Hamming weight\/}: 
$\delta_\mathbb{X}(d,r):=\delta_r(C_\mathbb{X}(d))$. 
\end{itemize}

The contents of this
paper are as follows. In Section~\ref{prelim-section} we 
present some of the results and terminology that will be needed
throughout the paper. If $F$ is a finite set of homogeneous polynomials of
$S\setminus\{0\}$ and $V_\mathbb{X}(F)$ is the set of zeros or
projective variety of $F$ in $\mathbb{X}$, over a finite field, we show
a degree formula for counting the number of points in 
$\mathbb{X}$ that are not in $V_{\mathbb{X}}(F)$: 
$$
|\mathbb{X}\setminus V_{\mathbb{X}}(F)|=\left\{
\begin{array}{cl}
\deg(S/(I(\mathbb{X})\colon(F)))&\mbox{if }(I(\mathbb{X})\colon(F))\neq
I(\mathbb{X}),\\ 
\deg(S/I(\mathbb{X}))&\mbox{if }(I(\mathbb{X})\colon
(F))=I(\mathbb{X}),
\end{array}
\right.
$$
and a degree formula for counting the zeros of $F$ in $\mathbb{X}$ 
(Lemmas~\ref{degree-formula-for-the-number-of-non-zeros} and 
\ref{degree-formula-for-the-number-of-zeros-proj}). These degree formulas
turn out to be useful in order to prove one of our main results
(Theorem~\ref{rth-footprint-lower-bound}).   

If $\mathbb{X}$ is a finite set of projective points over a
finite field and $I(\mathbb{X})$ is its vanishing ideal, we show
that $\delta_{I(\mathbb{X})}(d,r)$ is the $r$-th generalized Hamming weight
$\delta_\mathbb{X}(d,r)$ of the corresponding Reed-Muller-type 
code $C_\mathbb{X}(d)$ (Theorem~\ref{rth-min-dis-vi}). We introduce
the {\it Vasconcelos\/} function $\vartheta_I(d,r)$ of a graded ideal
$I$ (Definition~\ref{vasconcelos-function}) and show that 
$\vartheta_{I(\mathbb{X})}(d,r)$ is also equal to $\delta_\mathbb{X}(d,r)$
(Theorem~\ref{rth-min-dis-vi}). These two abstract algebraic formulations
of $\delta_\mathbb{X}(d,r)$ gives us a new tool to study generalized Hamming weights. 
One of our main results shows that ${\rm fp}_{I(\mathbb{X})}(d,r)$ 
is a lower bound for $\delta_\mathbb{X}(d,r)$ 
(Theorem~\ref{rth-footprint-lower-bound}). The 
{\it footprint matrix\/} $({\rm fp}_{I(\mathbb{X})}(d,r))$ 
and the {\it weight matrix} $(\delta_{\mathbb{X}}(d,r))$ of $I(\mathbb{X})$ are the 
matrices of size ${\rm
reg}(S/I(\mathbb{X}))\times\deg(S/I(\mathbb{X}))$
whose $(d,r)$-entries are ${\rm fp}_{I(\mathbb{X})}(d,r)$ and 
$\delta_{\mathbb{X}}(d,r)$, respectively
(see Remark~\ref{sarabia-vila}). 
In certain cases the $r$-th columns of these two matrices are equal 
(Example~\ref{footprint-matrix}, 
Theorem~\ref{application-to-rmtc-1}). The entries of each row of the weight matrix
$(\delta_{\mathbb{X}}(d,r))$ form an increasing sequence until they
stabilize \cite{wei}. For parameterized codes the entries of each column of the weight matrix
$(\delta_{\mathbb{X}}(d,r))$ form a decreasing sequence until they
stabilize \cite[Theorem~12]{camps-sarabia-sarmiento-vila}.

In Section~\ref{pncc-section} we introduce projective nested cartesian codes
\cite{carvalho-lopez-lopez}, a type of evaluation codes that 
generalize the classical projective Reed--Muller
codes \cite{lachaud,mercier-rolland,sorensen}. It is an interesting 
open problem to find an explicit formula for the
minimum distance of a projective nested
cartesian code. Using footprints and the integer inequality of
Lemma~\ref{pepe-vila-old} we show a uniform
upper bound for the number of zeros 
in a projective nested cartesian set $\mathcal{X}$ for a family of
homogeneous polynomials of fixed degree $d$, where $d$ is close to
the regularity of the ideal $I(\mathcal{X})$ (Theorem~\ref{sep10-16}). For projective spaces, 
this upper bound agrees
with the classical upper bound of S{\o}rensen and Serre
\cite[p.~1569]{sorensen}. 

As an application of our methods using generalized minimum
distance functions we were able to find a 
simple counterexample to a conjecture of Carvalho, Lopez-Neumann 
and L\'opez \cite{carvalho-lopez-lopez}, \cite[Conjecture~6.2]{hilbert-min-dis}
(Example~\ref{counterexample-nested-2-2-4}). 
In Sections~\ref{examples-section} and
\ref{procedures-section} we show some examples
and implementations in {\it Macaulay\/}$2$ \cite{mac2} that illustrate how some
of our results can be used in practice. A finite set of generators for the 
vanishing ideal of a projective space, over a finite field, was 
found by Mercier and Rolland \cite[Corollaire~2.1]{mercier-rolland}.
More generally, for the vanishing ideal of a projective nested
cartesian set a finite set of generators was determined in 
\cite[Lemma~2.4]{carvalho-lopez-lopez}. These results are especially
useful for computational purposes
(Examples~\ref{counterexample-nested-2-2-4} and \ref{dec4-15}).  

To show some other applications to algebraic coding theory we prove the following interesting
and non-trivial inequality.

\smallskip

\noindent {\bf Theorem~\ref{pepe-vila}}{\it\ Let $d\geq 1$ and 
$1\leq e_1\leq\cdots\leq e_m$ be integers. Suppose 
$1\leq a_i\leq e_i$ and $1\leq b_i\leq e_i$, for $i=1,\ldots,m$, are integers such that 
$d=\sum_ia_i=\sum_i
b_i$ and $a\neq b$. Then 
\[\pi(a,b)\geq
\left(\sum_{i=1}^m a_i-\sum_{i=k+1}^me_i-(k-2)\right)e_{k+1}\cdots
e_m-e_{k+2}\cdots e_m\] for $k=1,\ldots,m-1$, where 
$\pi(a,b)=\prod_{i=1}^ma_i+\prod_{i=1}^mb_i-\prod_{i=1}^m\min(a_i,b_i)$.}

\smallskip

We give two more applications. The first is the 
following explicit formula for the second generalized Hamming weight of an affine
cartesian code. 

\smallskip

\noindent {\bf Theorem~\ref{vila-pepe-sarabia-2}}{\it\ Let $A_i$, $i=1,\ldots,s-1$, be
subsets of $\mathbb{F}_q$ and let $\mathbb{X}\subset\mathbb{P}^{s-1}$ be the 
projective set $\mathbb{X}=[A_1\times\cdots\times
A_{s-1}\times\{1\}]$. If $d_i=|A_i|$ for $i=1,\ldots,s-1$ and
$2\leq d_1\leq \cdots\leq d_{s-1}$, then
$$
\delta_\mathbb{X}(d,2)=\left\{\hspace{-1mm}
\begin{array}{ll}\left(d_{k+1}-\ell+1\right)d_{k+2}\cdots
d_{s-1}-d_{k+3}\cdots d_{s-1}&\mbox{ if }
k<s-3,\\
\left(d_{k+1}-\ell+1\right)d_{k+2}\cdots d_{s-1}-1&\mbox{ if }
k=s-3,\\
\qquad \qquad d_{s-1}-\ell+1&\mbox{ if } k=s-2,
\end{array}
\right.
$$
where $0\leq k\leq s-2$ and $\ell$ are integers, 
$d=\sum_{i=1}^{k}\left(d_i-1\right)+\ell$, and $1\leq \ell \leq
d_{k+1}-1$.
}

\smallskip

Using this result one recover the case when $\mathbb{X}$ is a
projective torus in $\mathbb{P}^{s-1}$ 
\cite[Theorem~18]{camps-sarabia-sarmiento-vila} 
(Corollary~\ref{theorem3}).   
The second applications of this paper gives a combinatorial formula for
the second generalized Hamming weight 
of an affine cartesian code, which is quite different from the
corresponding formula of \cite[Theorem~5.4]{GHWCartesian}, and show 
that in this case the second generalized Hamming weight is equal to the
second generalized footprint. 

\smallskip 

\noindent {\bf Theorem~\ref{application-to-rmtc-1}}
{\it  Let $\mathcal{P}_d$ be the set of all pairs $(a,b)$,
$a,b$ in $\mathbb{N}^s$, $a=(a_i)$, $b=(b_i)$, 
such that $a\neq b$, $d=\sum_ia_i=\sum_ib_i$, $1\leq
a_i,b_i\leq d_i-1$ for $i=1,\ldots,n$, $n:=s-1$, $a_i\neq 0$ and $b_j\neq 0$
for some $1\leq i,j\leq n$. If $\mathbb{X}=[A_1\times\cdots\times
A_{n}\times\{1\}]$, with $A_i\subset\mathbb{F}_q$, $d_i=|A_i|$, and
$2\leq d_1\leq\cdots\leq d_{n}$, 
then
$$
{\rm fp}_{I(\mathbb{X})}(d,2)=\delta_\mathbb{X}(d,2)=\min\left\{P(a,b)\vert\,
(a,b)\in \mathcal{P}_d\right\}\ \mbox{ for }\ 
d\leq\textstyle\sum_{i=1}^n (d_i-1),
$$
where $P(a,b)=\prod_{i=1}^{n}(d_i-a_i)+
\prod_{i=1}^{n}(d_i-b_i)-\prod_{i=1}^{n}\min\{d_i-a_i,d_i-b_i\}$.}

\smallskip

In case the set of evaluation points $\mathbb{X}$ lie on an affine algebraic variety over
a finite field $\mathbb{F}_q$, the work done by Heijnen and Pellikaan
\cite{Pellikaan}, 
though formulated in a different
language, relates footprints and generalized Hamming weights and
introduce methods to study affine cartesian codes
(cf.~\cite[Section~7]{Pellikaan}). These methods were used in
\cite{GHWCartesian} to determine the generalized Hamming weights of
these codes. 
 
There is a nice combinatorial formula to compute 
the generalized Hamming weights of $q$-ary Reed-Muller codes  
\cite[Theorem~3.14]{Pellikaan}, and there is an easy to evaluate formula for the second
generalized Hamming weight of a projective torus 
\cite[Theorem~18]{camps-sarabia-sarmiento-vila}. There is also a
recent expression for the $r$-th generalized Hamming weight
of an affine cartesian code \cite[Theorem~5.4]{GHWCartesian}, which
depends on the $r$-th monomial in ascending lexicographic order of a
certain family of monomials. It is an interesting
problem to find alternative, easy to evaluate formulas for the $r$-th
generalized Hamming weight of                 
an affine cartesian code.

\smallskip

For all unexplained
terminology and additional information  we refer to 
\cite{BHer,CLO,Eisen} (for the theory of Gr\"obner bases, 
commutative algebra, and Hilbert functions), and
\cite{MacWilliams-Sloane,tsfasman} (for the theory of
error-correcting codes and linear codes). 

\section{Preliminaries}\label{prelim-section}

In this section we 
present some of the results that will be needed throughout the paper
and introduce some more notation. All results of this
section are well-known. To avoid repetitions, we continue to employ
the notations and 
definitions used in Section~\ref{intro-section}.

\paragraph{\bf Generalized Hamming weights} 
Let $K=\mathbb{F}_q$ be a finite field and let $C$ be a $[m,k]$ {\it linear
code} of {\it length} $m$ and {\it dimension} $k$.

The $r$-th {\it generalized Hamming weight\/} of $C$, denoted 
$\delta_r(C)$, is the size of the smallest support of an
$r$-dimensional subcode, that is,
$$
\delta_r(C):=\min\{|\chi(D)|\,\colon\, D\mbox{ is a linear subcode of
}C\mbox{ with }\dim_K(D)=r\}.
$$
\quad The {\it weight hierarchy\/} of $C$ is the sequence
$(\delta_1(C),\ldots,\delta_k(C))$. The integer $\delta_1(C)$ is
called the {\it minimum
distance\/} of $C$ and is denoted by $\delta(C)$. 
 According to \cite[Theorem~1,
Corollary~1]{wei}   
the weight hierarchy is an
increasing sequence 
$$
1\leq\delta_1(C)<\cdots<\delta_k(C)\leq m,
$$
and $\delta_r(C)\leq m-k+r$ for $r=1,\ldots,k$. For $r=1$ this is the
Singleton bound for the minimum distance. Notice that 
$\delta_r(C)\geq r$. 

Recall that the {\it support\/} $\chi(\beta)$ of a vector $\beta\in K^m$ 
is $\chi(K\beta)$, that is, $\chi(\beta)$ is the set of non-zero
entries of $\beta$.

\begin{lemma}\label{seminar} Let $D$ be a subcode of $C$ of dimension $r\geq 1$. If 
$\beta_1,\ldots,\beta_r$ is a $K$-basis for $D$ with
$\beta_i=(\beta_{i,1},\ldots,\beta_{i,m})$ for $i=1,\ldots,r$, then 
$\chi(D)=\cup_{i=1}^r\chi(\beta_i)$ and the number of elements of
$\chi(D)$ is the number of non-zero columns of the matrix:
$$   
\left[\begin{matrix}
\beta_{1,1}&\cdots&\beta_{1,i}&\cdots&\beta_{1,m}\\
\beta_{2,1}&\cdots&\beta_{2,i}&\cdots&\beta_{2,m}\\
\vdots&\cdots&\vdots&\cdots&\vdots\\
\beta_{r,1}&\cdots&\beta_{r,i}&\cdots&\beta_{r,m}
\end{matrix}\right].
$$
\end{lemma}

\paragraph{\bf Commutative algebra} Let $I\neq(0)$ be a graded ideal
of $S$ of Krull dimension $k$. The {\it Hilbert function} of $S/I$ is: 
$$
H_I(d):=\dim_K(S_d/I_d),\ \ \ d=0,1,2,\ldots,
$$
where $I_d=I\cap S_d$. By a theorem of Hilbert \cite[p.~58]{Sta1},
there is a unique polynomial 
$h_I(x)\in\mathbb{Q}[x]$ of 
degree $k-1$ such that $H_I(d)=h_I(d)$ for  $d\gg 0$. The
degree of the zero polynomial is $-1$.  

The {\it degree\/} or {\it multiplicity\/} of $S/I$ is the 
positive integer 
$$
\deg(S/I):=\left\{\begin{array}{ll}(k-1)!\, \lim_{d\rightarrow\infty}{H_I(d)}/{d^{k-1}}
&\mbox{if }k\geq 1,\\
\dim_K(S/I) &\mbox{if\ }k=0.
\end{array}\right.
$$ 

We will use the following 
multi-index notation: for $a=(a_1,\ldots,a_s)\in\mathbb{N}^s$, set
$t^a:=t_1^{a_1}\cdots t_s^{a_s}$. The multiplicative group of the
field $K$ is
denoted by 
$K^*$. As usual ${\rm ht}(I)$
will denote the height of the ideal $I$. By the {\rm dimension\/} 
of $I$ (resp. $S/I$) we mean the Krull dimension of $S/I$. The Krull dimension of
$S/I$ is denoted by $\dim(S/I)$.

One of the most useful and well-known facts about the degree is its additivity:

\begin{proposition}{\rm(Additivity of the degree
\cite[Proposition~2.5]{prim-dec-critical})}\label{additivity-of-the-degree}
If $I$ is an ideal of $S$ and 
$I=\mathfrak{q}_1\cap\cdots\cap\mathfrak{q}_m$ 
is an irredundant primary
decomposition, then
$$
\deg(S/I)=\sum_{{\rm ht}(\mathfrak{q}_i)={\rm
ht}(I)}\hspace{-3mm}\deg(S/\mathfrak{q}_i).$$
\end{proposition}

If $F\subset S$, the {\it ideal quotient\/} of $I$ with
respect to $(F)$ is given by $(I\colon(F))=\{h\in S\vert\, hF\subset I\}$.
An element $f$ is called a {\it zero-divisor\/} of $S/I$ if there is
$\overline{0}\neq \overline{a}\in S/I$ such that
$f\overline{a}=\overline{0}$, and $f$ is called {\it regular} on
$S/I$ if $f$ is not a zero-divisor. Thus $f$ is a zero-divisor if
and only if $(I\colon f)\neq I$. An associated prime of $I$ is a prime
ideal $\mathfrak{p}$ of $S$ of the form $\mathfrak{p}=(I\colon f)$
for some $f$ in $S$.
 
\begin{theorem}{\cite[Lemma~2.1.19,
Corollary~2.1.30]{monalg-rev}}\label{zero-divisors} If $I$ is an
ideal of $S$ and   
$I=\mathfrak{q}_1\cap\cdots\cap\mathfrak{q}_m$ is 
an irredundant primary decomposition with ${\rm
rad}(\mathfrak{q}_i)=\mathfrak{p}_i$, then the set of zero-divisors
$\mathcal{Z}(S/I)$  of $S/I$ is equal to
$\bigcup_{i=1}^m\mathfrak{p}_i$, 
and $\mathfrak{p}_1,\ldots,\mathfrak{p}_m$ are the associated primes of
$I$. 
\end{theorem}

\begin{definition}
The {\it regularity\/} of $S/I$, denoted 
${\rm reg}(S/I)$, is the least integer $r_0\geq 0$ such that
$H_I(d)$ is equal to $h_I(d)$ for $d\geq r_0$.  
\end{definition}

\paragraph{\bf The footprint of an ideal} Let  $\prec$ be a monomial
order on $S$ and let $(0)\neq I\subset S$ be an ideal. If $f$ is a non-zero 
polynomial in $S$, the {\it leading
monomial\/} of $f$ 
is denoted by ${\rm in}_\prec(f)$. The {\it initial ideal\/} of $I$, denoted by
${\rm in}_\prec(I)$,  is the monomial ideal given by 
$${\rm in}_\prec(I)=(\{{\rm in}_\prec(f)|\, f\in I\}).
$$ 

A monomial $t^a$ is called a 
{\it standard monomial\/} of $S/I$, with respect 
to $\prec$, if $t^a$ is not in the ideal ${\rm in}_\prec(I)$. 
A polynomial $f$ is called {\it standard\/} if
$f\neq 0$ and $f$ is a
$K$-linear combination of standard monomials. 
The set of standard monomials, denoted $\Delta_\prec(I)$, is called the {\it
footprint\/} of $S/I$. The image of the standard polynomials of
degree $d$, under the canonical map $S\mapsto S/I$, 
$x\mapsto \overline{x}$, is equal to $S_d/I_d$, and the 
image of $\Delta_\prec(I)$ is a basis of $S/I$ as a $K$-vector space. 
This is a classical result of Macaulay (for a modern approach
see \cite[Chapter~5]{CLO}). In
particular, if $I$ is graded, then $H_I(d)$ is the number of standard
monomials of degree $d$.

\begin{lemma}{\cite[p.~2]{carvalho}}\label{nov6-15} Let $I\subset S$ be an ideal generated by
$\mathcal{G}=\{g_1,\ldots,g_r\}$, then
$$
\Delta_\prec(I)\subset\Delta_\prec({\rm in}_\prec(g_1),\ldots,{\rm
in}_\prec(g_r)).
$$
\end{lemma}

\paragraph{\bf Vanishing ideal of a finite set} 
The {\it projective space\/} of 
dimension $s-1$
over the field $K$ is denoted $\mathbb{P}^{s-1}$. It is usual to denote the 
equivalence class of $\alpha$ by $[\alpha]$. 

For a given a subset $\mathbb{X}\subset\mathbb{P}^{s-1}$ define 
$I(\mathbb{X})$, the {\it vanishing ideal\/} of $\mathbb{X}$, 
as the ideal generated by the homogeneous polynomials 
in $S$ that vanish at all points of $\mathbb{X}$, and 
given a graded ideal $I\subset S$ 
define its {\it zero set\/} relative to $\mathbb{X}$ as  
$$V_\mathbb{X}(I)=\left\{[\alpha]\in \mathbb{X}\vert\, 
f(\alpha)=0\,\,  
\forall f\in I\, \mbox{ homogeneous} \right\}.
$$ 
In particular, if $f\in S$ is homogeneous, the zero set
$V_\mathbb{X}(f)$ of $f$ is the set of all $[\alpha]\in \mathbb{X}$
such that $f(\alpha)=0$, that is $V_\mathbb{X}(f)$ is the set of zeros
of $f$ in $\mathbb{X}$. 

\begin{lemma}\label{primdec-ix-a} Let $\mathbb{X}$ be a finite
subset of $\mathbb{P}^{s-1}$, let $[\alpha]$ be a point in
$\mathbb{X}$  
with $\alpha=(\alpha_1,\ldots,\alpha_s)$
and $\alpha_k\neq 0$ for some $k$, and let
$I_{[\alpha]}$ be the vanishing ideal of $[\alpha]$. Then $I_{[\alpha]}$ is a prime ideal, 
\begin{equation*}
I_{[\alpha]}=(\{\alpha_kt_i-\alpha_it_k\vert\, k\neq i\in\{1,\ldots,s\}),\
\deg(S/I_{[\alpha]})=1,\,  
\end{equation*}
${\rm ht}(I_{[\alpha]})=s-1$, 
and $I(\mathbb{X})=\bigcap_{[\beta]\in{\mathbb{X}}}I_{[\beta]}$ is the primary
decomposition of $I(\mathbb{X})$. 
\end{lemma}

\begin{definition}\label{projectivetorus-def} The set 
$\mathbb{T}=\{[(x_1,\ldots,x_s)]\in\mathbb{P}^{s-1}\vert\, x_i\in
K^*,\, \forall\, i\}$ is called a {\it projective
torus\/}.
\end{definition}

\section{Computing the number of points of a 
variety}\label{computing-zeros-of-varieties-section}
In this section we give a degree
formula to compute the number of solutions of a system of 
homogeneous polynomials over any given finite set of
points in a projective space over a field.

\begin{lemma}\label{vila-vivares-mar10-17} Let $\mathbb{X}$ be a finite subset of 
$\mathbb{P}^{s-1}$ over a field $K$. If $F=\{f_1,\ldots,f_r\}$ is a
set of homogeneous polynomials of $S\setminus\{0\}$, 
then $V_\mathbb{X}(F)=\emptyset$ if and only
if $(I(\mathbb{X})\colon(F))=I(\mathbb{X})$.
\end{lemma}

\begin{proof} $\Rightarrow$) We proceed by contradiction assuming 
that $I(\mathbb{X})\subsetneq(I(\mathbb{X})\colon(F))$. Pick a homogeneous polynomial
$g$ such that $gf_i\in I(\mathbb{X})$ for all $i$ and $g\notin I(\mathbb{X})$. 
Then there is $[\alpha]$ in $\mathbb{\mathbb{X}}$ such that $g(\alpha)\neq 0$. 
Thus $f_i(\alpha)=0$ for all $i$, that is, $[\alpha]\in
V_\mathbb{\mathbb{X}}(F)$, a contradiction.

$\Leftarrow$) We can write $\mathbb{X}=\{[P_1],\ldots,[P_m]\}$ and
$I(\mathbb{X})=\cap_{i=1}^m\mathfrak{p}_i$, where $\mathfrak{p}_i$ is
equal to $I_{[P_i]}$, the vanishing ideal of $[P_i]$. We proceed by
contradiction assuming that $V_\mathbb{X}(F)\neq\emptyset$. Pick
$[P_i]$ in $V_\mathbb{X}(F)$. For simplicity of notation assume that
$i=1$. Notice that $(\mathfrak{p}_1\colon(F))=(1)$. Therefore
\[
\bigcap_{i=1}^m\mathfrak{p}_i=I(\mathbb{X})=(I(\mathbb{X})\colon(F))=
\bigcap_{i=1}^m(\mathfrak{p}_i\colon(F))
=\bigcap_{i=2}^m(\mathfrak{p}_i\colon(F))\subset\mathfrak{p}_1.
\]
Hence $\mathfrak{p}_i\subset(\mathfrak{p}_i\colon(F))\subset\mathfrak{p}_1$ 
for some $i\geq 2$, see \cite[p.~74]{monalg-rev}. 
Thus $\mathfrak{p}_i=\mathfrak{p}_1$, a
contradiction.
\end{proof} 

An ideal $I\subset S$ is called {\it unmixed\/} 
if all its associated primes have the same height, and $I$ is called
{\it radical\/} if $I$ is equal to its radical. The radical of $I$ is
denoted by ${\rm rad}(I)$.

\begin{lemma}\label{degree-formula-for-the-number-of-non-zeros}
Let $\mathbb{X}$ be a finite subset of 
$\mathbb{P}^{s-1}$ over a field $K$ and let $I(\mathbb{X})\subset S$ be its
vanishing ideal. If 
$F=\{f_1,\ldots,f_r\}$ is a set of homogeneous polynomials of
$S\setminus\{0\}$, then 
$$
|\mathbb{X}\setminus V_{\mathbb{X}}(F)|=\left\{
\begin{array}{cl}
\deg(S/(I(\mathbb{X})\colon(F)))&\mbox{if }\, (I(\mathbb{X})\colon(F))\neq
I(\mathbb{X}),\\ 
\deg(S/I(\mathbb{X}))&\mbox{if }\, (I(\mathbb{X})\colon (F))=I(\mathbb{X}).
\end{array}
\right.
$$
\end{lemma}

\begin{proof} Let $[P_1],\ldots,[P_m]$ be the points of $\mathbb{X}$
with $m=|\mathbb{X}|$, and let $[P]$ be a point in
$\mathbb{X}$ with $P=(\alpha_1,\ldots,\alpha_s)$
and $\alpha_k\neq 0$ for some $k$. Then the vanishing ideal $I_{[P]}$
of $[P]$ is a prime ideal of height $s-1$, 
\begin{equation*}
I_{[P]}=(\{\alpha_kt_i-\alpha_it_k\vert\, k\neq i\in\{1,\ldots,s\}),\
\deg(S/I_{[P]})=1,
\end{equation*}
and $I(\mathbb{X})=\bigcap_{i=1}^mI_{[P_i]}$ is a primary
decomposition (see Lemma~\ref{primdec-ix-a}). 

Assume that $(I(\mathbb{X})\colon (F))\neq I(\mathbb{X})$. We set
$I=I(\mathbb{X})$ and $\mathfrak{p}_i=I_{[P_i]}$ for $i=1,\ldots,m$. 
Notice that $(\mathfrak{p}_j\colon f_i)=(1)$ if and only if
$f_i\in\mathfrak{p}_j$ if and only if $f_i(P_j)=0$. Then
\begin{equation*}
(I\colon(F))=\bigcap_{i=1}^r(I\colon
f_i)=\left(\bigcap_{f_1(P_j)\neq 0}\mathfrak{p}_j\right)\bigcap
\cdots\bigcap \left(\bigcap_{f_r(P_j)\neq 0}\mathfrak{p}_j\right)
=\bigcap_{[P_j]\notin V_\mathbb{X}(F)}\mathfrak{p}_j.
\end{equation*}

Therefore,  by the additivity of the degree of
Proposition~\ref{additivity-of-the-degree}, we get that 
$\deg(S/(I\colon(F)))$ is equal to $|\mathbb{X}\setminus
V_\mathbb{X}(F)|$. If $(I(\mathbb{X})\colon(F))=I(\mathbb{X})$, then 
 $V_\mathbb{X}(F)=\emptyset$ (see Lemma~\ref{vila-vivares-mar10-17}).
 Thus $|V_\mathbb{X}(F)|=0$ and the required formula follows 
because $|\mathbb{X}|=\deg(S/I(\mathbb{X}))$. 
\end{proof}

\begin{lemma}\label{jul11-15} Let $I\subset S$ be a radical unmixed graded ideal. 
If $F=\{f_1,\ldots,f_r\}$ is a set of homogeneous polynomials of
$S\setminus\{0\}$, $(I\colon (F))\neq I$, and $\mathcal{A}$ is
the set of all associated primes
of $S/I$ that contain $F$, then ${\rm ht}(I)={\rm ht}(I,F)$,
$\mathcal{A}\neq\emptyset$ and 
$$
\deg(S/(I,F))=\sum_{\mathfrak{p}\in\mathcal{A}}\deg(S/\mathfrak{p}).
$$
\end{lemma}

\begin{proof} As $I\subsetneq (I\colon (F))$, there is $g\in
S\setminus{I}$ such that $g(F)\subset I$. Hence the ideal 
$(F)$ is contained in the set of zero-divisors of $S/I$. Thus, by
Theorem~\ref{zero-divisors} and since $I$ is unmixed, $(F)$ is contained in an associated 
prime ideal $\mathfrak{p}$ of $S/I$ of
height ${\rm ht}(I)$. Thus $I\subset(I,F)\subset\mathfrak{p}$, and
consequently  ${\rm ht}(I)={\rm ht}(I,F)$. Therefore the set of
associated primes of $(I,F)$ of height equal to ${\rm ht}(I)$ is not empty and is equal to 
 $\mathcal{A}$. There is an irredundant primary decomposition
\begin{equation}\label{jul10-15}
(I,F)=\mathfrak{q}_1\cap\cdots\cap\mathfrak{q}_n\cap\mathfrak{q}_{n+1}'\cap
\cdots\cap\mathfrak{q}_t',
\end{equation}
where ${\rm rad}(\mathfrak{q}_i)=\mathfrak{p_i}$,
$\mathcal{A}=\{\mathfrak{p}_1,\ldots,\mathfrak{p}_n\}$, and ${\rm
ht}(\mathfrak{q}_i')>{\rm ht}(I)$ for $i>n$. We may assume that the
associated primes of $S/I$ are
$\mathfrak{p}_1,\ldots,\mathfrak{p}_m$ with $n\leq m$. Since
$I$ is a radical ideal, we get that $I=\cap_{i=1}^m\mathfrak{p}_i$.
Next we show the following equality:
\begin{equation}\label{jul10-15-1}
\mathfrak{p}_1\cap\cdots\cap\mathfrak{p}_m=
\mathfrak{q}_1\cap\cdots\cap\mathfrak{q}_n\cap\mathfrak{q}_{n+1}'\cap\cdots\cap\mathfrak{q}_t'
\cap\mathfrak{p}_{n+1}\cap\cdots\cap\mathfrak{p}_m.
\end{equation}
The inclusion ``$\supset$'' is clear because
$\mathfrak{q}_i\subset\mathfrak{p}_i$ for $i=1,\ldots,n$. The 
inclusion ``$\subset$'' follows by noticing that the right hand side of
Eq.~(\ref{jul10-15-1}) is equal to 
$(I,f)\cap\mathfrak{p}_{n+1}\cap\cdots\cap\mathfrak{p}_m$, and 
consequently it contains $I=\cap_{i=1}^m\mathfrak{p}_i$. Notice that
${{\rm rad}}(\mathfrak{q}_j')=\mathfrak{p}_j'\not\subset\mathfrak{p}_i$ for
all $i,j$ and $\mathfrak{p}_j\not\subset\mathfrak{p}_i$ for $i\neq j$.
Hence localizing Eq.~(\ref{jul10-15-1}) 
at the prime ideal $\mathfrak{p}_i$ for $i=1,\ldots,n$, 
we get that $\mathfrak{p}_i=I_{\mathfrak{p}_i}\cap
S=(\mathfrak{q}_i)_{\mathfrak{p}_i}\cap S=\mathfrak{q}_i$ for
$i=1,\ldots,n$. Using Eq.~(\ref{jul10-15}), together with the additivity of the
degree of Proposition~\ref{additivity-of-the-degree}, the required equality follows.
\end{proof}

\begin{lemma}\label{degree-formula-for-the-number-of-zeros-proj}
Let $\mathbb{X}$ be a finite subset of 
$\mathbb{P}^{s-1}$ over a field $K$ and let $I(\mathbb{X})\subset S$ be its
vanishing ideal. If 
$F=\{f_1,\ldots,f_r\}$ is a set of homogeneous polynomials of
$S\setminus\{0\}$, then the number of points of $V_\mathbb{X}(F)$ is given by 
$$
|V_{\mathbb{X}}(F)|=\left\{
\begin{array}{cl}
\deg(S/(I(\mathbb{X}),F))&\mbox{if }\, (I(\mathbb{X})\colon(F))\neq
I(\mathbb{X}),\\ 
0&\mbox{if }\, (I(\mathbb{X})\colon(F))=I(\mathbb{X}).
\end{array}
\right.
$$
\end{lemma}

\begin{proof} Let $[P_1],\ldots,[P_m]$ be the points of $\mathbb{X}$
with $m=|\mathbb{X}|$. The vanishing ideal $I_{[P_i]}$ of $[P_i]$ is a
 prime ideal of height $s-1$, $\deg(S/I_{[P_i]})=1$, and
 $I(\mathbb{X})=\bigcap_{i=1}^mI_{[P_i]}$ (see Lemma~\ref{primdec-ix-a}). 

Assume that $(I(\mathbb{X})\colon(F))\neq I(\mathbb{X})$. Let
$\mathcal{A}$ be the set of all $I_{[P_i]}$ that contain the
set $F$. Notice that $f_j\in I_{[P_i]}$ if and only if $f_j(P_i)=0$. Then 
$[P_i]$ is in $V_\mathbb{X}(F)$ if and only if 
$F\subset I_{[P_i]}$. Thus $[P_i]$ is in $V_\mathbb{X}(F)$ if and only
if $I_{[P_i]}$ is in $\mathcal{A}$. Hence, 
by Lemma~\ref{jul11-15}, we get 
$$
|V_\mathbb{X}(F)|=\sum_{[P_i]\in
V_\mathbb{X}(F)}\deg(S/I_{[P_i]})=\sum_{F\subset I_{[P_i]}}\deg(
S/I_{[P_i]})=\deg(S/(I(\mathbb{X}),F)).
$$

Assume that $(I(\mathbb{X})\colon F)=I(\mathbb{X})$. Then, by 
Lemma~\ref{vila-vivares-mar10-17}, $V_\mathbb{X}(f)=\emptyset$ and
$|V_\mathbb{X}(f)|=0$.
\end{proof}

\begin{proposition} If $\mathbb{X}$ is a finite subset of
$\mathbb{P}^{s-1}$, then 
$$ 
\deg(S/I(\mathbb{X}))=\deg(S/(I(\mathbb{X}),F))+\deg(S/(I(\mathbb{X})\colon(F))).
$$
\end{proposition}

\begin{proof} It follows from Lemmas~\ref{degree-formula-for-the-number-of-non-zeros} and
\ref{degree-formula-for-the-number-of-zeros-proj}.
\end{proof}

\section{Generalized minimum distance function of a graded
ideal}\label{mdf-section}

In this part we study the generalized minimum distance function of a graded
ideal. To
avoid repetitions, we continue to employ 
the notations and definitions used in Sections~\ref{intro-section} and
\ref{prelim-section}. 

\begin{lemma}\label{degree-initial-footprint} Let $I\subset S$ be an
unmixed graded ideal and let $\prec$ be 
a monomial order. If $F$ is a finite set of homogeneous polynomials of
$S$ and $(I\colon(F))\neq I$, then
$$
\deg(S/(I,F))\leq\deg(S/({\rm
in}_\prec(I),{\rm in}_\prec(F)))\leq\deg(S/I),
$$
and $\deg(S/(I,F))<\deg(S/I)$ if $I$ is an unmixed radical ideal and
$(F)\not\subset I$.
\end{lemma}

\begin{proof} To simplify notation we set $J=(I,F)$, 
$L=({\rm in}_\prec(I),{\rm in}_\prec(F))$, and
$F=\{f_1,\ldots,f_r\}$. 
We denote the Krull
dimension of $S/I$ by $\dim(S/I)$. 
Recall that $\dim(S/I)=\dim(S)-{\rm ht}(I)$. First we show that $S/J$ and
$S/L$ have Krull dimension equal to $\dim(S/I)$. As
$I\subsetneq(I\colon F)$, all elements of $F$ are zero divisors of
$S/I$. Hence, as $I$ is unmixed, there is an
associated prime ideal $\mathfrak{p}$ of $S/I$ such that
$(F)\subset\mathfrak{p}$ and $\dim(S/I)=\dim(S/\mathfrak{p})$. Since 
$I\subset J\subset\mathfrak{p}$, we get that $\dim(S/J)$ is 
$\dim(S/I)$.  
Since $S/I$ and $S/{\rm
in}_\prec(I)$ have the same Hilbert function, and so does
$S/\mathfrak{p}$ and $S/{\rm in}_\prec(\mathfrak{p})$, we obtain
$$
\dim(S/{\rm in}_\prec(I))=\dim(S/I)=\dim(S/\mathfrak{p})=\dim(S/{\rm
in}_\prec(\mathfrak{p})).
$$
Hence, taking heights in the inclusions ${\rm in}_\prec(I)\subset L\subset{\rm
in}_\prec(\mathfrak{p})$, we obtain ${\rm ht}(I)={\rm ht}(L)$. 

Pick a Gr\"obner basis $\mathcal{G}=\{g_1,\ldots,g_r\}$ of $I$. Then
$J$ is generated by $\mathcal{G}\cup F$ and by Lemma~\ref{nov6-15}
one has the inclusions
\begin{eqnarray*}
& &\Delta_\prec(J)=\Delta_\prec(I,F)\subset\Delta_\prec({\rm in}_\prec(g_1),\ldots,{\rm
in}_\prec(g_r),{\rm in}_\prec(F))=\\
& &\ \ \ \ \  \Delta_\prec({\rm in}_\prec(I),{\rm
in}_\prec(F))=\Delta_\prec(L)\subset \Delta_\prec({\rm in}_\prec(g_1),\ldots,{\rm
in}_\prec(g_r))=\Delta_\prec(I).
\end{eqnarray*}
Thus $\Delta_\prec(J)\subset \Delta_\prec(L)\subset \Delta_\prec(I)$.
Recall that $H_I(d)$, the Hilbert function of $I$ at $d$, is the number of standard
monomials of degree $d$. Hence $H_J(d)\leq H_L(d)\leq H_I(d)$ for
$d\geq 0$. If $\dim(S/I)$ is equal to $0$, then 
$$
\deg(S/J)=\sum_{d\geq 0}H_J(d)\leq \deg(S/L)=\sum_{d\geq 0}H_L(d)\leq
\deg(S/I)=\sum_{d\geq 0}H_I(d).
$$

Assume now that $\dim(S/I)\geq 1$. 
By a theorem of Hilbert \cite[p.~58]{Sta1}, 
$H_J$, $H_L$, $H_I$ are polynomial functions of degree
equal to $k=\dim(S/I)-1$ (see \cite[Theorem~4.1.3]{BHer}). 
Thus 
$$k!\lim_{d\rightarrow\infty} H_J(d)/d^k\leq
k!\lim_{d\rightarrow\infty} H_L(d)/d^k\leq
k!\lim_{d\rightarrow\infty} H_I(d)/d^k,$$
that is, $\deg(S/J)\leq\deg(S/L)\leq\deg(S/I)$. 

If $I$ is an unmixed radical ideal and $(F)\not\subset I$, then there is at
least one minimal prime that does not contains $(F)$. Hence, by
Lemma~\ref{jul11-15}, it follows that $\deg(S/(I,F))<\deg(S/I)$.
\end{proof}

\begin{corollary}\label{poly-bounds-initial} Let $\mathbb{X}$ be a
finite subset of $\mathbb{P}^{s-1}$, let $I(\mathbb{X})\subset
S$ be its vanishing ideal, and let $\prec$ be a monomial order. 
If $F$ is a finite set of homogeneous polynomials of
$S$ and $(I(\mathbb{X})\colon(F))\neq I(\mathbb{X})$, then
$$
|V_{\mathbb{X}}(F)|=\deg(S/(I(\mathbb{X}),F))\leq\deg(S/({\rm
in}_\prec(I(\mathbb{X})),{\rm in}_\prec(F)))\leq\deg(S/I(\mathbb{X})),
$$
and $\deg(S/(I(\mathbb{X}),F))<\deg(S/I(\mathbb{X}))$ if
$(F)\not\subset I(\mathbb{X})$.
\end{corollary}

\begin{proof} It follows from
Lemmas~\ref{degree-formula-for-the-number-of-zeros-proj} and \ref{degree-initial-footprint}. 
\end{proof}

\begin{lemma}\label{mar14-17} 
Let $\mathbb{X}=\{[P_1],\ldots,[P_m]\}$ be a finite subset of
$\mathbb{P}^{s-1}$ and let $D$ be a linear subspace of $C_\mathbb{X}(d)$ of 
dimension $r\geq 1$. The following hold.
\begin{itemize}
\item[(i)] There are 
$\overline{f}_1,\ldots,\overline{f}_r$ linearly independent elements
of $S_d/I_d$ such that $D=\oplus_{i=1}^rK\beta_i$, where 
$\beta_i$ is $(f_i(P_1),\ldots,f_i(P_m))$, and the support 
$\chi(D)$ of $D$ is equal to $\cup_{i=1}^r\chi(\beta_i)$.
\item[(ii)] $|\chi(D)|=|\mathbb{X}\setminus
V_\mathbb{X}(f_1,\ldots,f_r)|$. 
\item[(iii)] $\delta_r(C_\mathbb{X}(d))=\min\{|\mathbb{X}\setminus
V_\mathbb{X}(F)|:\, F=\{f_i\}_{i=1}^r\subset S_d,\,
\{\overline{f}_i\}_{i=1}^r
\mbox{linearly independent over } K\}$.
\end{itemize}
\end{lemma}

\begin{proof} (i): This part follows from Lemma~\ref{seminar} and using that the evaluation map
${\rm ev}_d$ induces an isomorphism between $S_d/I_d$ and 
$C_\mathbb{X}(d)$.

(ii): Consider the matrix $A$ with rows $\beta_1, \ldots,\beta_r$.
Notice that the $i$-th column of $A$ is not zero 
if and only if $[P_i]$ is in $\mathbb{X}\setminus
V_\mathbb{X}(f_1,\ldots,f_r)$. It suffices to observe that the number of non-zero 
columns of $A$ is $|\chi(D)|$ (see Lemma~\ref{seminar}).

(iii): This follows from part (ii) and using the definition of the
$r$-th generalized Hamming weight of $C_\mathbb{X}(d)$ (see
Section~\ref{prelim-section}).
\end{proof}

\begin{definition}\label{vasconcelos-function} If $I\subset S$ is a
graded ideal, the {\it Vasconcelos function\/} of $I$ is the
function 
$\vartheta_I\colon \mathbb{N}_+\times\mathbb{N}_+\rightarrow \mathbb{N}$ given by 
$$
\vartheta_I(d,r):=\left\{\begin{array}{ll}\min\{\deg(S/(I\colon(F)))\vert\,
F\in\mathcal{F}_{d,r}\}&\mbox{if }\mathcal{F}_{d,r}\neq\emptyset,\\
\deg(S/I)&\mbox{if\ }\mathcal{F}_{d,r}=\emptyset.
\end{array}\right.
$$
\end{definition}

\begin{theorem}\label{rth-min-dis-vi} 
Let $K$ be a field and let $\mathbb{X}$ be a finite subset of
$\mathbb{P}^{s-1}$. If  $|\mathbb{X}|\geq 2$ and
$\delta_\mathbb{X}(d,r)$ is the $r$-th generalized Hamming weight of
$C_\mathbb{X}(d)$, then 
$$\delta_\mathbb{X}(d,r)=\delta_{I(\mathbb{X})}(d,r)=\vartheta_{I(\mathbb{X})}(d,r)\
\mbox{ for }d\geq 1\mbox{ and }1\leq r\leq H_{I(\mathbb{X})}(d),$$
and $\delta_\mathbb{X}(d,r)=r$ for $d\geq {\rm reg}(S/I(\mathbb{X}))$.
\end{theorem}

\begin{proof} If $\mathcal{F}_{d,r}=\emptyset$, then using
Lemmas~\ref{degree-formula-for-the-number-of-non-zeros}, 
\ref{degree-formula-for-the-number-of-zeros-proj}, and
\ref{mar14-17} we get that $\delta_\mathbb{X}(d,r)$, 
$\delta_{I(\mathbb{X})}(d,r)$, and $\vartheta_{I(\mathbb{X})}(d,r)$ are equal to
$\deg(S/I(\mathbb{X}))=|\mathbb{X}|$. Assume that
$\mathcal{F}_{d,r}\neq \emptyset$ and set $I=I(\mathbb{X})$. 
Using Lemma~\ref{mar14-17} and the formula for $V_\mathbb{X}(F)$ of 
Lemma~\ref{degree-formula-for-the-number-of-zeros-proj}, we
obtain
\begin{eqnarray*}
\delta_\mathbb{X}(d,r)&\stackrel{(\ref{mar14-17})}{=}&\min\{|\mathbb{X}\setminus
V_\mathbb{X}(F)|\colon F\in
\mathcal{F}_{d,r}\}
\stackrel{(\ref{degree-formula-for-the-number-of-zeros-proj})}{=}
|\mathbb{X}|-\max\{\deg(S/(I,F))\vert\,
F\in \mathcal{F}_{d,r}\}\\
&=&\deg(S/I)-\max\{\deg(S/(I,F))\vert\,
F\in \mathcal{F}_{d,r}\}=\delta_{I}(d,r),\mbox{ and}\\
\delta_\mathbb{X}(d,r)&\stackrel{(\ref{mar14-17})}{=}&\min\{|\mathbb{X}\setminus
V_\mathbb{X}(F)|\colon F\in
\mathcal{F}_{d,r}\}
\stackrel{(\ref{degree-formula-for-the-number-of-non-zeros})}{=}
\min\{\deg(S/(I\colon(F)))\vert\,
F\in \mathcal{F}_{d,r}\}=\vartheta_{I}(d,r).
\end{eqnarray*}
In these equalities we used the fact that
$\deg(S/I(\mathbb{X}))=|\mathbb{X}|$. As $H_{I}(d)=|\mathbb{X}|$ for
$d\geq{\rm reg}(S/I)$, using the generalized Singleton bound for the
generalized Hamming weight and the fact that the weight hierarchy is an
increasing sequence we obtain that $\delta_\mathbb{X}(d,r)=r$ for
$d\geq {\rm reg}(S/I(\mathbb{X}))$ (see \cite[Theorem~1,
Corollary~1]{wei}).
\end{proof}

\begin{remark}\label{sarabia-vila} Let $\mathbb{X}$ be a finite set of projective points
over a field $K$. The following hold.
\begin{itemize}
\item[(a)] $r\leq\delta_\mathbb{X}(d,r)\leq |\mathbb{X}|$ 
for $d\geq 1$ and $1\leq r\leq H_{I(\mathbb{X})}(d)$. This follows 
from the fact that the weight hierarchy is an
increasing sequence (see \cite[Theorem~1]{wei}).
\item[(b)] If $d\geq {\rm reg}(S/I(\mathbb{X}))$, then
$C_\mathbb{X}(d)=K^{|\mathbb{X}|}$ and $\delta_\mathbb{X}(d,r)=r$ for
$1\leq r\leq |\mathbb{X}|$.
\item[(c)] If $C_\mathbb{X}(d)$ is non-degenerate, i.e., for each
$1\leq i\leq |\mathbb{X}|$ there is $\alpha\in C_\mathbb{X}(d)$ whose
$i$-th entry is non-zero, then
$\delta_\mathbb{X}(d,H_\mathbb{X}(d))=|\mathbb{X}|$.
\item[(d)] If $r>H_{I(\mathbb{X})}(d)$, then
$\mathcal{F}_{d,r}=\emptyset$ and
$\delta_{I(\mathbb{X})}(d,r)=|\mathbb{X}|$.
\end{itemize}
\end{remark}

\begin{lemma}\label{regular-elt-in} 
Let $\prec$ be a monomial order, let $I\subset S$ be an ideal, let
$F=\{f_1,\ldots,f_r\}$ be a set of polynomial of $S$ of positive
degree, and let ${\rm in}_\prec(F)=\{{\rm in}_\prec(f_1),\ldots,{\rm
in}_\prec(f_r)\}$ be the set of initial terms of $F$. 
If $({\rm in}_\prec(I)\colon({\rm in}_\prec(F)))={\rm in}_\prec(I)$, then
$(I\colon(F))=I$. 
\end{lemma}

\begin{proof} Let $g$ be a polynomial of $(I\colon(F))$, that is,   
$gf_i\in I$ for $i=1,\ldots,r$. It
suffices to show that $g\in I$. Pick a Gr\"obner basis
$g_1,\ldots,g_n$ of $I$. Then, by the division algorithm
\cite[Theorem~3, p.~63]{CLO}, we can write 
$g=\sum_{i=1}^nh_ig_i+h$, 
where $h=0$ or $h$ is a finite sum of monomials not in 
${\rm in}_\prec(I)=({\rm in}_\prec(g_1),\ldots,{\rm
in}_\prec(g_n))$. We need only
show that $h=0$. If $h\neq 0$, then $hf_i$ is in $I$ 
and ${\rm in}_\prec(h){\rm in}_\prec(f_i)$ is in the ideal ${\rm
in}_\prec(I)$ for $i=1,\ldots,r$ . Hence ${\rm in}_\prec(h)$ is
in  $({\rm in}_\prec(I)\colon({\rm in}_\prec(F)))$. Therefore, by
hypothesis, ${\rm in}_\prec(h)$ is in the ideal ${\rm in}_\prec(I)$, 
a contradiction.
\end{proof}

Let $\prec$ be a monomial order and let $\mathcal{F}_{\prec,d,r}$ be
the set of all subsets 
$F=\{f_1,\ldots,f_r\}$ of $S_d$ such that $(I\colon(F))\neq I$, 
$f_i$ is a standard polynomial for all $i$, 
$\overline{f}_1,\ldots,\overline{f}_r$ are linearly independent over
the field $K$, and ${\rm in}_\prec(f_1),\ldots,{\rm in}_\prec(f_r)$
are distinct monomials. Let $F=\{f_1,\ldots,f_r\}$ be a set of
standard polynomials. It is not hard to see that
$\overline{f}_1,\ldots,\overline{f}_r$ are linearly independent over
the field $K$ if ${\rm in}_\prec(f_1),\ldots,{\rm in}_\prec(f_r)$
are distinct monomials.

The next result is useful for computations
with {\it Macaulay\/}$2$ \cite{mac2} (see 
Procedure~\ref{procedure-min-dis-fooprint}).

\begin{proposition}\label{march27-17} The generalized minimum distance function of $I$
is given by 
$$
\delta_I(d,r)=\left\{\begin{array}{ll}\deg(S/I)-\max\{\deg(S/(I,F))\vert\,
F\in\mathcal{F}_{\prec, d,r}\}&\mbox{if }\mathcal{F}_{\prec,d,r}\neq\emptyset,\\
\deg(S/I)&\mbox{if\ }\mathcal{F}_{\prec,d,r}=\emptyset.
\end{array}\right.
$$
\end{proposition}

\begin{proof} Take $F=\{f_1,\ldots,f_r\}$ in $\mathcal{F}_{d,r}$. By
the division algorithm any $f_i$ can be written as $f_i=p_i+h_i$,
where $p_i$ is in $I_d$ and $h_i$ is a $K$-linear combination of standard
monomials of degree $d$. Setting $H=\{h_1,\ldots,h_r\}$, notice that 
$(I\colon(F))=(I\colon(H))$, $(I,F)=(I,H)$,
$\overline{f}_i=\overline{h}_i$ for $i=1,\ldots,r$. Thus
$H\in\mathcal{F}_{d,r}$, that is, we may assume 
that $f_1,\ldots,f_r$ are standard polynomials. Setting
$KF=Kf_1+\cdots+Kf_r$, we claim that there is
a set $G=\{g_1,\ldots,g_r\}$ consisting of homogeneous standard 
polynomials of $S/I$ of degree $d$ such that $KF=KG$, ${\rm in}_\prec(g_1),\ldots,{\rm
in}_\prec(g_r)$ distinct monomials, and ${\rm
in}_\prec(f_i)\succeq{\rm in}_\prec(g_i)$ for all $i$. We proceed by
induction on $r$. The case $r=1$ is clear. Assume that $r>1$.
Permuting the $f_i$'s if necessary we may assume that ${\rm
in}_\prec(f_1)\succeq\cdots\succeq{\rm
in}_\prec(f_r)$. If ${\rm in}_\prec(f_1)\succ{\rm in}_\prec(f_2)$, the
claim follows applying the induction hypothesis to $f_2,\ldots,f_r$.
If ${\rm in}_\prec(f_1)={\rm in}_\prec(f_2)$, there is $k\geq 2$ such
that ${\rm in}_\prec(f_1)={\rm in}_\prec(f_i)$ for $i\leq k$ and 
${\rm in}_\prec(f_1)\succ {\rm in}_\prec(f_i)$ for $i>k$. We set
$h_i=f_1-f_i$ for $i=2,\ldots,k$ and $h_i=f_i$ for $i=k+1,\ldots,r$.
Notice that ${\rm in}_\prec(f_1)\succ h_i$ for $i\geq 2$ and that 
$h_2,\ldots,h_r$ are standard monomials of degree $d$ which are linearly
independent over $K$. Hence the claim follows applying the induction
hypothesis to $H=\{h_2,\ldots,h_r\}$. The required expression for
$\delta_I(d,r)$ follows readily using Theorem~\ref{rth-min-dis-vi}.
\end{proof}

Let $I$ be a graded ideal, let $\Delta_\prec^p(I)_d$ be the set of
standard polynomials of $S/I$ of degree $d$, and let
$\mathcal{F}^p_{\prec,d,r}$ be the set of all subsets of
$\Delta_\prec^p(I)_d$ with $r$ elements. If $K=\mathbb{F}_q$ is a
finite field, then
$$
|\Delta_\prec^p(I)_d|=q^{H_I(d)}-1\ \mbox{ and }\
|\mathcal{F}^p_{\prec,d,r}|=\binom{q^{H_I(d)}-1}{r}.
$$
\quad Thus computing $\delta_I(d,r)$ is very hard because one has to
determine which of the polynomials in $\mathcal{F}^p_{\prec,d,r}$ are
in $\mathcal{F}_{\prec,d,r}$, and then compute the corresponding
degrees. To compute
${\rm fp}_I(d,r)$ is much simpler because we only need to determine the set $\Delta_\prec(I)_{d,r}$
of all subsets of $\Delta_\prec(I)_d$ with $r$ elements, and one has
$$
|\Delta_\prec(I)_{d,r}|=\binom{H_I(d)}{r},
$$
which is much smaller than the size of $\mathcal{F}^p_{\prec,d,r}$.

We come to one of our main results.

\begin{theorem}\label{rth-footprint-lower-bound} 
Let $K$ be a field, let $\mathbb{X}$ be a finite subset of
$\mathbb{P}^{s-1}$, and let $\prec$ be a monomial order. If  $|\mathbb{X}|\geq 2$ and
$\delta_\mathbb{X}(d,r)$ is the $r$-th generalized Hamming weight of
$C_\mathbb{X}(d)$, then 
$${\rm fp}_{I(\mathbb{X})}(d,r)\leq \delta_\mathbb{X}(d,r)\ \mbox{ for
}d\geq 1\mbox{ and }1\leq r\leq H_{I(\mathbb{X})}(d).$$
\end{theorem}

\begin{proof} This follows from Theorem~\ref{rth-min-dis-vi},
Lemma~\ref{regular-elt-in}, and Proposition~\ref{march27-17}.
\end{proof}

\section{Projective nested cartesian codes}\label{pncc-section}

In this section we introduce projective nested cartesian codes, a type of evaluation codes that 
generalize the classical projective Reed--Muller
codes \cite{lachaud,mercier-rolland,sorensen}, and give a lower bound for the minimum distance
of some of these codes. 

Let $K=\mathbb{F}_q$ be a finite field, let $A_1,\ldots,A_s$ be a collection of subsets 
of $K$, and let 
$$
\mathcal{X}=[A_1\times\cdots\times A_s]
$$
be the image of $A_1\times\cdots\times A_s\setminus\{0\}$ under the map 
$K^s\setminus\{0\}\rightarrow\mathbb{P}^{s-1}$, $x\rightarrow [x]$. 

\begin{definition}{\cite{carvalho-lopez-lopez}}\label{pncc}
The set $\mathcal{X}$ is called a
{\it projective nested cartesian set} if  
\begin{itemize}
\item[(i)] $\{0,1\}\subset A_i$ for $i=1,\ldots,s$, 
\item[(ii)] $a/b\in A_j$ for  $1\leq i<j\leq s$, 
$a\in A_j$, $0\neq b\in A_i$, and 
\item[(iii)] $d_1\leq\cdots\leq d_s$, where $d_i=|A_i|$ for
$i=1,\ldots,s$.  
\end{itemize}
\quad If $\mathcal{X}$ is a projective nested cartesian set and $C_\mathcal{X}(d)$ is its 
corresponding $d$-th projective Reed-Muller-type code, we call 
$C_{\mathcal{X}}(d)$ a {\it projective nested cartesian
code}. 
\end{definition}

The next conjecture is not true as will be shown in
Example~\ref{counterexample-nested-2-2-4}.

\begin{conjecture}{\rm(Carvalho, Lopez-Neumann, and L\'opez 
\cite{carvalho-lopez-lopez},
\cite[Conjecture~6.2]{hilbert-min-dis})}\label{carvalho-lopez-lopez-conjecture}
Let $C_\mathcal{X}(d)$ be the $d$-th projective nested
cartesian code on the set
$\mathcal{X}=[A_1\times\cdots\times A_s]$ with $d_i=|A_i|$ for
$i=1,\ldots,s$. Then its minimum distance is given by 
$$
\delta_\mathcal{X}(d,1)=\left\{\hspace{-1mm}
\begin{array}{ll}\left(d_{k+2}-\ell+1\right)d_{k+3}\cdots d_s&\mbox{ if }
d\leq \sum\limits_{i=2}^{s}\left(d_i-1\right),\\
\qquad \qquad 1&\mbox{ if } d\geq
\sum\limits_{i=2}^{s}\left(d_i-1\right)+1,
\end{array}
\right.
$$
where $0\leq k\leq s-2$ and $\ell$ are the unique integers such that 
$d=\sum_{i=2}^{k+1}\left(d_i-1\right)+\ell$ and $1\leq \ell \leq
d_{k+2}-1$. 
\end{conjecture}

This conjecture is true in some 
special cases \cite[Theorem~3.8]{carvalho-lopez-lopez}, 
which includes the classical projective Reed-Muller codes
\cite{sorensen}. The minimum distance of $C_\mathcal{X}(d)$ proposed in
Conjecture~\ref{carvalho-lopez-lopez-conjecture} 
is in fact the minimum distance of a certain evaluation linear code
\cite[Corollary~6.9]{hilbert-min-dis}. 

In what follows $\mathcal{X}=[A_1\times\cdots\times A_s]$ denotes a
projective nested cartesian set, with $d_i=|A_i|$, and $C_\mathcal{X}(d)$ is its 
corresponding projective nested cartesian code. Throughout this
section $\prec$ is the lexicographical order on $S$, with 
$t_1\prec\cdots \prec t_s$, and ${\rm in}_\prec(I(\mathcal{X}))$ is the
initial ideal of $I(\mathcal{X})$.

\begin{lemma}{\cite[Lemma~5.6]{hilbert-min-dis}}\label{pepe-vila-old} 
Let $1\leq e_1\leq\cdots\leq e_m$ and $0\leq b_i\leq e_i-1$
for $i=1,\ldots,m$ be integers. Then 
\begin{equation}\label{aug-18-15-old}
\prod_{i=1}^m(e_i-b_i)\geq\left(\sum_{i=1}^k(e_i-b_i)-(k-1)-
\sum_{i=k+1}^mb_i\right)e_{k+1}\cdots e_m
\end{equation}
for $k=1,\ldots,m$, where $e_{k+1}\cdots e_m=1$ and
$\sum_{i=k+1}^mb_i=0$ if $k=m$.
\end{lemma}

\begin{lemma}\label{sep5-16}  Let $g$
be a homogeneous polynomial of $S$ of degree $1\leq d\leq \sum_{i=1}^{s-1}(d_i-1)$
and let $[P]$ be a point of $\mathbb{P}^{s-1}$. If $g$ vanishes at all
points of $\mathcal{X}\setminus\{[P]\}$, then $g(P)=0$. 
\end{lemma}

\begin{proof} Assume that $g(P)\neq 0$ and set 
$\mathcal{X}^*=A_1\times\cdots\times A_s$. One can write
the point $[P]$ in standard form 
$[P]=(0,\ldots,0,p_k,\ldots,p_s)$ with $p_k=1$. Notice that 
$V_{\mathcal{X}^*}(g)$, the zero set of $g$ in $\mathcal{X}^*$, has
cardinality equal to $|\mathcal{X}^*|-(d_k-1)$. Setting 
$f_i=\prod_{\gamma\in A_i}(t_i-\gamma)$ for $1\leq i\leq s$ and
applying \cite[Lemma~2.3]{cartesian-codes}, we get
$$
I(X^*)=(f_1,\ldots,f_s)\ \ \mbox{ and }\ \  
{\rm in}_\prec(I(\mathcal{X}^*))=(t_1^{d_1},\ldots,t_s^{d_s}).
$$
\quad Therefore, by the division
algorithm, we can write 
$$
g=h_1f_1+\cdots+h_sf_s+f,
$$
where $f(0)=0$, $\deg(f)\leq\deg(g)$, and $\deg_{t_i}(f)\leq d_i-1$
for all $i$. Thus
\begin{equation*}
|V_{\mathcal{X}^*}(f)|=|V_{\mathcal{X}^*}(g)|=d_1\cdots d_s-(d_k-1).
\end{equation*}

The initial term of $f$ has the form $t^a=t_1^{a_1}\cdots
t_s^{a_s}$, where $a_i\leq d_i-1$ for all $i$. Hence  
\begin{eqnarray*}
d_1\cdots
d_s-(d_k-1)&=&|V_{\mathcal{X}^*}(f)|=\deg(S/(I(\mathcal{X}^*),f))\\
&\leq&\deg(S/({\rm in}_\prec(I(\mathcal{X}^*),{\rm in}_\prec(f))=
\deg(S/(t_1^{d_1},\ldots, t_s^{d_s},t^a)))\\
&=&d_1\cdots d_s-(d_1-a_1)\cdots(d_s-a_s).
\end{eqnarray*}

Thus $d_k-1\geq (d_1-a_1)\cdots(d_s-a_s)$. Applying
Lemma~\ref{pepe-vila-old} with $m=k=s$, $e_i=d_i$, and $b_i=a_i$ for all
$i$, we get
\begin{eqnarray*}
(d_1-a_1)\cdots(d_s-a_s)&\geq&\sum_{i=1}^s(d_i-a_i)-(s-1)\\
&=&\sum_{i=1}^{s-1}(d_i-1)-\sum_{i=1}^sa_i+d_s\geq d_s\geq d_k.
\end{eqnarray*}
Thus $(d_1-a_1)\cdots(d_s-a_s)\geq d_k$, a contradiction.
\end{proof} 

The next result gives a uniform upper bound for the number of zeros
in a projective nested cartesian set for a family of
homogeneous polynomials of fixed degree $d$, where d is within a
certain range, 
and a corresponding lower
bound for $\delta_\mathcal{X}(d,1)$.

\begin{theorem}\label{sep10-16} Let $f$ be a polynomial of degree $d$
that does 
not vanish at all points of $\mathcal{X}$. If $d=\sum_{i=2}^{s-1}(d_i-1)+\ell$ and
$1\leq \ell\leq d_1-1 $, then
$$
|V_\mathcal{X}(f)|\leq\deg(S/I(\mathcal{X}))-(d_1-\ell+1).
$$
In particular $\delta_\mathcal{X}(d,1)\geq d_1-\ell+1\geq 2$.
\end{theorem}

\begin{proof} We proceed by contradiction assuming that
$|V_\mathcal{X}(f)|\geq |\mathcal{X}|-(d_1-\ell)$. If $n$ is the 
number of elements of $\mathcal{X}\setminus V_\mathcal{X}(f)$, then 
$n\leq d_1-\ell$. Let $[P_1],\ldots,[P_n]$ be the points of 
$\mathcal{X}\setminus V_\mathcal{X}(f)$. For each $1\leq i\leq n-1$, 
pick $h_i\in S_1$ such that $h_i(P_i)=0$ and $h_i(P_n)\neq 0$. 
Setting $g=fh_1\cdots h_{n-1}$, one has 
$$
\deg(g)=(n-1)+d=(n-1)+\sum_{i=2}^{s-1}(d_i-1)+\ell<
d_1+\sum_{i=2}^{s-1}(d_i-1),
$$
$g(P_n)\neq 0$, and $g$ vanishes at all
points of $\mathcal{X}\setminus\{[P_n]\}$. This contradicts 
Lemma~\ref{sep5-16}. 
\end{proof}

\section{Examples}\label{examples-section}

In this section we show some examples that illustrate how some
of our results can be used in practice. In particular we give a
counterexample to Conjecture~\ref{carvalho-lopez-lopez-conjecture}. 

\begin{example}\label{counterexample-nested-2-2-4}\rm
Let $K$ be the field $\mathbb{F}_{4}$, let $\mathcal{X}$ be the 
projective nested cartesian set
\begin{center}
$\mathcal{X}=[A_1\times A_2 \times A_3]\subset\mathbb{P}^2$
\end{center} 
where $A_1=\{0,1\}, A_2=\{0,1\}, A_3=\mathbb{F}_4$, and let 
$I=I(\mathcal{X})$ be the vanishing ideal of $\mathcal{X}$. The ideal
$I$ is generated by
$t_{1}t_2^2-t_1^{2}t_2,\, t_1t_3^4-t_1^4t_3,\, t_2t_3^4-t_2^4t_3$,
${\rm reg}(S/I)=5$, and $\deg(S/I)=13$. This
follows from \cite[Lemma~2.4]{carvalho-lopez-lopez}. 
Using Procedure~\ref{procedure-min-dis-fooprint}, we obtain 
the following table with the basic parameters of $C_\mathcal{X}(d)$:

\begin{eqnarray*}
\hspace{-11mm}&&\left.
\begin{array}{c|c|c|c|c|c|c}
d & 1 & 2 & 3 & 4 & 5 &\cdots\\
   \hline
 |\mathcal{X}| & 13 & 13 & 13 & 13 & 13 & \cdots
 \\ 
   \hline
 H_\mathcal{X}(d)    \    & 3 & 6  & 9 & 12 & 13 &\cdots\\ 
   \hline
 \delta_{\mathcal{X}}(d,1) & 8 & 4 & 3 & 1 & 1 &\cdots \\ 
\hline
 {\rm fp}_{I(\mathcal{X})}(d,1) & 8 & 4 & 3 & 1 & 1 &\cdots \\ 
\end{array}
\right.
\end{eqnarray*}

A direct way to see that $\delta_\mathcal{X}(4,1)=1$ is to observe 
that the polynomial $f=t_3(t_3^3-t_2^3-t_1^3+t_1^2t_2) $ vanishes at all points
of $\mathcal{X}\setminus\{[e_3]\}$ and $f(e_3)=1$, where
$e_3=(0,0,1)$. The ideal $I$ is Geil-Carvalho \cite{min-dis-ci}, 
that is, $\delta_\mathcal{X}(d,1)={\rm fp}_{I(\mathcal{X})}(d,1)$ 
for $d\geq 1$.  Using the notation of 
Conjecture~\ref{carvalho-lopez-lopez-conjecture}, 
the values of $\delta_{\mathcal{X}}(d,1)$ according to this conjecture are
given by the following table:  
\begin{eqnarray*}
\hspace{-11mm}&&\left.
\begin{array}{c|c|c|c|c|c|c}
d & 1 & 2 & 3 & 4 & 5 &\cdots\\
   \hline
 k & 0 & 1 & 1 & 1 & &
 \\ 
   \hline
 \ell    \    & 1 & 1 & 2 & 3 & &\\ 
   \hline
 \delta_{\mathcal{X}}(d,1) & 8 & 4 & 3 & 2 & 1 & \cdots \\ 
\end{array}
\right.
\end{eqnarray*}
\quad Thus the conjecture fails in degree
$d=4$. 
\end{example}

\begin{example}\label{example-computing-points}\rm 
Let $K$ be the field $\mathbb{F}_{4}$, let $\mathcal{X}$ be the 
projective nested cartesian set
\begin{center}
$\mathcal{X}=[A_1\times A_2 \times A_3]\subset\mathbb{P}^2$, 
\end{center} 
where $A_1=\{0,1\}, A_2=\{0,1\}, A_3=\mathbb{F}_4$ and let 
$I=I(\mathcal{X})$ be the vanishing ideal of $\mathcal{X}$. 
Consider the following pairs
of polynomials $F=\{f_1,f_2\}$ of degree $d$: 
$$
\begin{array}{llll}
\mbox{Case } d=1\colon&\mbox{Case } d=2\colon&\mbox{Case } d=3\colon&\mbox{Case }
d=4\colon\cr
f_1=t_1-t_2,&f_1=(t_1-t_2)(t_1-t_3),&f_1=(t_1-t_2)(t_1-t_3)t_2,&f_1=(t_1-t_2)(t_1-t_3)t_2^2,\cr
f_2=t_1-t_3.&f_2=(t_1-t_2)t_2.&f_2=(t_1-t_2)t_2^2.&f_2=(t_1-t_2)(t_2-t_3)t_2t_3.
\end{array}
$$
\quad If $V_\mathcal{X}(F)$ is 
the variety in $\mathcal{X}$ defined by $F=\{f_1,f_2\}$, using 
Procedure~\ref{procedure-degree-formula-for-the-number-of-zeros-proj}
we obtain: 
$$
\begin{array}{ll}
\mbox{Case } d=1\colon\  |V_\mathcal{X}(F)|=\deg(S/(I,F))=1,&\mbox{Case }
d=2\colon\
|V_\mathcal{X}(F)|=\deg(S/(I,F))=6,\cr
\mbox{Case } d=3\colon\ |V_\mathcal{X}(F)|=\deg(S/(I,F))=9,&\mbox{Case }
d=4\colon\ 
|V_\mathcal{X}(F)|=\deg(S/(I,F))=10.
\end{array}
$$
\end{example}

\begin{example}\label{footprint-matrix} Let $({\rm
fp}_{I(\mathcal{X})}(d,r))$ and 
$(\delta_{\mathcal{X}}(d,r))$ be the {\it footprint
matrix\/} and the {\it weight matrix} of the ideal $I(\mathcal{X})$ of 
Example~\ref{counterexample-nested-2-2-4}. 
These matrices are of size $5\times 13$ because the regularity and the
degree of $S/I(\mathcal{X})$ are $5$ and $13$, respectively. Using 
Procedure~\ref{procedure-footprint-matrix} we obtain: 
$$
({\rm fp}_{I(\mathcal{X})}(d,r))=\left[\begin{array}{ccccccccccccc}
8& 12 & 13 & \infty & \infty & \infty & \infty & \infty& \infty& \infty&\infty& \infty&\infty\\
4& 7 & 8 & 11 & 12 & 13 & \infty  & \infty & \infty & \infty &\infty & \infty &\infty \\
3& 4 & 6 & 7 & 8 & 10 & 11 & 12& 13& \infty &\infty & \infty &\infty \\
1 & 3& 4 & 5 & 6 & 7 & 8 & 9& 10& 11&12& 13&\infty \\
1& 2 & 3 &4 & 5 & 6 & 7 & 8& 9& 10&11& 12&13
\end{array}\right].
$$

\quad If $r>H_\mathcal{X}(d)$, then $\mathcal{M}_{\prec,
d,r}=\emptyset$ and the  $(d,r)$-entry of this matrix is
equal to $|\mathcal{X}|$, but in this case we prefer to write $\infty$ for 
computational reasons. Therefore, by
Theorem~\ref{rth-footprint-lower-bound}, we obtain $({\rm
fp}_{I(\mathcal{X})}(d,r))\leq(\delta_\mathcal{X}(d,r))$. By 
Example~\ref{counterexample-nested-2-2-4} one has ${\rm
fp}_{I(\mathcal{X})}(d,1)=\delta_\mathcal{X}(d,1)$ for $d\geq 1$ and
by Example~\ref{example-computing-points} it follows readily that ${\rm
fp}_{I(\mathcal{X})}(d,2)=\delta_\mathcal{X}(d,2)$ for $d\geq 1$.
Using that the generalized footprint is a lower bound for the
generalized Hamming weight we have verified that ${\rm
fp}_{I(\mathcal{X})}(d,r)$ is equal to $\delta_\mathcal{X}(d,r)$ for all $d,r$.
\end{example}

\begin{example}\label{dec4-15}
Let $\mathcal{X}=\mathbb{P}^2$ be the projective space over 
the field $\mathbb{F}_2$.  The vanishing ideal 
$I=I(\mathcal{X})$ is the ideal of $S=\mathbb{F}_2[t_1,t_2,t_3]$ generated by the binomials 
$t_1t_2^2-t_1^2t_2,\, t_1t_3^2-t_1^2t_3,\, t_2^2t_3-t_2t_3^2$ 
\cite[Corollaire~2.1]{mercier-rolland}. If $S$ has the GRevLex order
$\prec$, 
adapting Procedure~\ref{procedure-min-dis-fooprint} we get 
\begin{eqnarray*}
\hspace{-11mm}&&\left.
\begin{array}{c|c|c|c|c}
d & 1 & 2 & 3 & \cdots\\
   \hline
{\deg}(S/I) & 7 & 7 & 7 &\cdots
 \\ 
   \hline
 H_I(d)    \    & 3 & 6  & 7& 
 \cdots\\ 
   \hline
 \delta_\mathcal{X}(d,1) &4& 2& 1& \cdots\\ 
 \hline
 {\rm fp}_I(d,1) &4& 1& 1& \cdots\\ 
\end{array}
\right.
\end{eqnarray*}
\end{example}

\section{Procedures for {\it Macaulay\/}$2$}\label{procedures-section} 
In this section we present the procedures used to compute the examples
of Section~\ref{examples-section}. Some of these procedures work for
graded ideals and provide a tool to study 
generalized minimum distance functions.

\begin{procedure}\label{procedure-min-dis-fooprint} Computing the
minimum distance  and
the footprint with {\it Macaulay\/}$2$ \cite{mac2} using Theorem~\ref{rth-min-dis-vi} and
Proposition~\ref{march27-17}. The next procedure corresponds to
Example~\ref{counterexample-nested-2-2-4}.
\begin{verbatim}
q=4
G=GF(q,Variable=>a)
S=G[t3,t2,t1,MonomialOrder=>Lex]
I=ideal(t1*t2^2-t1^2*t2,t1*t3^4-t1^4*t3,t2^4*t3-t2*t3^4)
M=coker gens gb I, degree M, regularity M
init=ideal(leadTerm gens gb I)
H=(d)->hilbertFunction(d,M), apply(1..regularity(M),H)
h=(d)->degree M - max apply(apply(apply(apply(
toList (set(0,a,a^2,a^3))^**(hilbertFunction(d,M))-(set{0})^**(
hilbertFunction(d,M)), toList),x->basis(d,M)*vector x),
z->ideal(flatten entries z)), x-> if not quotient(I,x)==I then degree
ideal(I,x) else 0)--h(d) is the minimum distance in degree d
apply(1..regularity(M)-1,h)
f=(x)-> if not quotient(init,x)==init then degree ideal(init,x) else 0
fp=(d) ->degree M -max apply(flatten entries basis(d,M),f)
--fp(d) is the footprint in degree d
apply(1..regularity(M),fp)
f=t3*(t3^3-t2^3-t1^3+t1^2*t2), degree ideal(I,f)
\end{verbatim}
\end{procedure}

\begin{procedure}\label{procedure-degree-formula-for-the-number-of-zeros-proj}
Computing the number of solutions of a system of 
homogeneous polynomials over any given set of
projective points over a finite field with 
 {\it Macaulay\/}$2$ \cite{mac2} using the degree formula of 
Lemma~\ref{rth-min-dis-vi}. The next procedure corresponds to
Example~\ref{example-computing-points}.  
\begin{verbatim}
q=4
G=GF(q,Variable=>a)
S=G[t3,t2,t1,MonomialOrder=>Lex]
I=ideal(t1^2*t2-t1*t2^2,t1*t3^4-t1^4*t3,t2^4*t3-t2*t3^4)
f1=t1-t2, f2=t1-t3, quotient(I,ideal(f1,f2))==I
degree (I+ideal(f1,f2))
f1=(t1-t2)*(t1-t3), f2=(t1-t2)*t2, quotient(I,ideal(f1,f2))==I
degree (I+ideal(f1,f2))
f1=(t1-t2)*(t1-t3)*t2, f2=(t1-t2)*t2^2, quotient(I,ideal(f1,f2))==I
degree (I+ideal(f1,f2))
f1=(t1-t2)*(t1-t3)*t2^2, f2=(t1-t2)*(t2-t3)*t2*t3
quotient(I,ideal(f1,f2))==I
degree (I+ideal(f1,f2))
\end{verbatim}
\end{procedure}

\begin{procedure}\label{procedure-footprint-matrix} Computing the 
footprint matrix with {\it Macaulay\/}$2$ \cite{mac2}. This procedure
corresponds to Example~\ref{footprint-matrix}. It can be 
applied to any vanishing ideal $I$ to obtain the entries of the 
matrix $({\rm fp}_{I}(d,r))$ and is reasonably fast.  
\begin{verbatim}
q=4
G=GF(q,Variable=>a)
S=G[t3,t2,t1,MonomialOrder=>Lex]
I=ideal(t1^2*t2-t1*t2^2, t1*t3^4-t1^4*t3, t2^4*t3-t2*t3^4)
M=coker gens gb I, regularity M, degree M
init=ideal(leadTerm gens gb I)
er=(x)-> if not quotient(init,x)==init then degree ideal(init,x) else 0
fpr=(d,r)->degree M - max apply(apply(apply(
subsets(flatten entries basis(d,M),r),toSequence),ideal),er)
g=(r)->apply(sort toList(set(1..regularity(M))**set{r}),fpr)
--g(r) is the r-th column of the footprint matrix 
\end{verbatim}
\end{procedure}

\section{An integer inequality}

For $a:a_1,\ldots,a_m$ and $b:b_1,\ldots,b_m$ sequences in
$\Z^+=\{1,2,\ldots\}$ we
define
\[\pi(a,b):=\prod_{i=1}^ma_i+\prod_{i=1}^mb_i-\prod_{i=1}^m\min(a_i,b_i).\]
\begin{lemma}\label{lemma1-pepe} Let $a_1,a_2,b_1,b_2\in\Z^+$.
Set $a_1'=\min(a_1,a_2)$ and $a_2'=\max(a_1,a_2)$. Then
\[\min(a_1,b_1)\min(a_2,b_2)\leq\min(a_1',b_1')\,\min(a_2',b_2').\]
\end{lemma}
\begin{proof}
It is an easy case-by-case verification of $4!$ possible cases.
\end{proof}

\begin{lemma}\label{lemma2-pepe} Let $a:a_1,\ldots,a_m$ and
$b:b_1,\ldots,b_m$ be sequences in 
$\Z^+$. Suppose:
\begin{itemize}
 \item[$(i)$] $r<s$, $a_r>a_s$. Set $a_r'=a_s$,
$a_s'=a_r$, $a_i'=a_i$ for $i\neq r,s$; and $b_r'=\min(b_r,b_s)$,
$b_s'=\max(b_r,b_s)$, $b_i'=b_i$ for $i\neq r,s$. Then $\pi(a,b)\geq
\pi(a',b')$.
 \item[$(ii)$] $r<s$, $b_r=a_r\leq a_s<b_s$. Set $a_r'=a_r-1$,
$a_s'=a_s+1$, $a_i'=a_i$ for $i\neq r,s$. Then $\pi(a,b)\geq\pi(a',b)$.
 \item[$(iii)$] $r<s$, $b_r<a_r\leq a_s$. Set $a_r'=a_r-1$,
$a_s'=a_s+1$, $a_i'=a_i$ for $i\neq r,s$. Then $\pi(a,b)\geq\pi(a',b)$.
 \item[$(iv)$] $r<s$, $a_r<a_s$, $b_r=a_s, b_s=a_r$,
$b_i=a_i$ for $i\neq r,s$, $h:=a_s-a_r\geq 2$. Set $b_r'=a_r+1$, $b_s'=a_s-1$,
$b_i'=a_i$ for $i\neq r,s$. Then $\pi(a,b)\geq \pi(a,b')$.
\end{itemize}
\end{lemma}

\noindent {\it Proof.} We verify all cases by direct substitution of
$a'$ and $b'$ into $\pi(a,b)$.
\begin{eqnarray*}
&&\hspace{-1cm}(i)\ \ \pi(a,b)-\pi(a',b')=\prod a_i+\prod b_i-\prod a_i'-\prod
b_i'+\prod\min(a_i',b_i')
-\prod\min(a_i,b_i)\\
&=&(\min(a_r',b_r')\min(a_s',b_s')- \min(a_r,b_r)\min(a_s,b_s))\prod_{i\neq
r,s}\min(a_i,b_i)\geq 0.\ \ (Lemma~\ref{lemma1-pepe})
\end{eqnarray*}
\begin{eqnarray*}
&&\hspace{-1.5cm}(ii)\ \ \pi(a,b)-\pi(a',b)=\prod a_i-\prod
a_i'+\prod\min(a_i',b_i)
-\prod\min(a_i,b_i)\\
&=&(a_ra_s-(a_r-1)(a_s+1))\prod_{i\neq
r,s}a_i\\
&&+(\min(a_r',b_r)\min(a_{s}',b_{s})-\min(a_r,b_r)\min(a_{s},b_{s}))\prod_{i\neq
r,s}\min(a_i,b_i)\\
&=&(a_s-a_r+1)\prod_{i\neq r,s}a_i+((a_r-1)(a_s+1)-a_ra_s)\prod_{i\neq
r,s}\min(a_i,b_i)\\
&=&(a_s-a_r+1)\left(\prod_{i\neq r,s}a_i-\prod_{i\neq
r,s}\min(a_i,b_i)\right)\geq 0.
\end{eqnarray*}
\begin{eqnarray*}
&&\hspace{-1.5cm}(iii)\ \ \pi(a,b)-\pi(a',b)=\prod a_i-\prod
a_i'+\prod\min(a_i',b_i)
-\prod\min(a_i,b_i)\\
&=&(a_ra_s-(a_r-1)(a_s+1))\prod_{i\neq
r,s}a_i\\
&&+(\min(a_r',b_r)\min(a_{s}',b_{s})-\min(a_r,b_r)\min(a_{s},b_{s}))\prod_{i\neq
r,s}\min(a_i,b_i)\\
&=&(a_s-a_r+1)\prod_{i\neq
r,s}a_i+b_r(\min(a_s+1,b_s)-\min(a_s,b_s))\prod_{i\neq r,s}\min(a_i,b_i)\geq 0.
\end{eqnarray*}
For the last inequality note that $\min(a_s+1,b_s)-\min(a_s,b_s)=0$ or $1$.
\begin{eqnarray*}
(iv)\ \ \pi(a,b)-\pi(a,b')&=&\prod b_i-\prod b_i'+\prod\min(a_i,b_i')-\prod\min(a_i,b_i)\\
&=&(a_ra_s-(a_r+1)(a_s-1)+a_r(a_s-1)-a_r^2)\prod_{i\neq
r,s}a_i\\
&=&(a_r-1)(h-1)\prod_{i\neq r,s}a_i\geq 0. \qed
\end{eqnarray*}

\begin{lemma}\label{basic-ineq}
If $a_1,\ldots,a_r$ are positive integers, then
$a_1\cdots a_r\geq(a_1+\cdots+a_r)-(r-1)$.
\end{lemma}

\begin{proof} It follows by induction on $r$. 
\end{proof}

\begin{lemma}\label{lemma3-pepe} 
Let $1\leq e_1\leq\cdots\leq e_m$ and $1\leq a_i,b_i\leq e_i$,
for $i=1,\ldots,m$ be integers. Suppose $a_i=b_i=1$ for $i<r$, $a_i=b_i=e_i$ for
$i>r+1:=s$, $1\leq a_i,b_i\leq e_i$ for $i=r,s$, with $a_r+a_s=b_r+b_s$ and
$(a_r,a_s)\neq(b_r,b_s)$. If $b_r\leq a_s$ and $b_s=a_s-1$, then
\begin{equation}\label{pepe-ineq1}
\pi(a,b)\geq
\left(\sum_{i=1}^m a_i-\sum_{i=k+1}^me_i-(k-2)\right)e_{k+1}\cdots
e_m-e_{k+2}\cdots e_m
\end{equation}
for $i=1,\ldots,m-1$, where $e_{k+2}\cdots
e_m=1$ when $k=m-1$.
\end{lemma}

\begin{proof} Set $\sigma=\sum_{i=1}^m a_i-\sum_{i=k+1}^me_i-(k-2)$. 
Since $b_s(b_r-a_r)=a_s-1$, one has the equality
\begin{equation}\label{pepe-ineq2}
\pi(a,b)=(a_ra_s+b_rb_s-a_rb_s)\prod_{i=r+2}^me_i=(a_ra_s+a_s-1)\prod_{i=r+2}^me_i.
\end{equation}

Case $k+1<r$: The integer $\sigma$ can be rewritten as 
\[\sigma=k+(1-e_{k+1})+\cdots+(1-e_{r-1})+(a_r-e_r)+(a_s-e_s)-(k-2).\]
\quad Since $a_r<b_r\leq e_r$, it holds that $a_r-e_r\leq -1$, and hence $\sigma\leq
1$. If $\sigma\leq 0$, Eq.~(\ref{pepe-ineq1}) trivially follows (because the left hand side
is positive and the right hand side would be negative). So we may assume
$\sigma=1$. This assumption implies that $e_{k+1}=1$ because
$a_r<b_r\leq e_r$. Then the right hand side
of Eq.~(\ref{pepe-ineq1}) is
\[(\sigma)e_{k+1}\cdots e_m-e_{k+2}\cdots e_m=(e_{k+1}-1)e_{k+2}\cdots e_r=0.\]

\smallskip 

Case $k+1=r$: The integer $ \sigma$ can be rewritten as
\[\sigma=k+(a_r-e_r)+(a_s-e_s)-(k-2).\]
By the same reason as above, we may assume $\sigma=1$. This assumption implies
$a_r=e_r-1$ and $a_s=e_s$. Then, by Eq.~(\ref{pepe-ineq2}), we obtain
that Eq.~(\ref{pepe-ineq1}) is equivalent to

\[(e_re_s-1)\prod_{i=r+2}^m e_i\geq\left(\sigma\right)e_{r}\cdots e_m-e_{r+1}\cdots e_m,\]
which reduces to $e_re_s-1\geq\left(1\right)e_re_s-e_s$, or equivalently,
$e_s\geq 1$.

\smallskip 

Case $k+1=r+1$: We can rewrite $\sigma$ as
\[\sigma=(k-1)+a_r+(a_s-e_s)-(k-2)=a_r+(a_s-e_s)+1.\] 
Then, using Eq.~(\ref{pepe-ineq2}), we obtain that
Eq.~(\ref{pepe-ineq1}) is equivalent to
\[(a_ra_s+a_s-1)\prod_{i=r+2}^m e_i\geq\left(\sigma\right)e_{r+1}\cdots e_m-e_{r+2}\cdots e_m,\]
which reduces to $a_ra_s+a_s\geq\left(a_r+a_s-e_s+1\right)e_{s}$, 
or equivalently,
\[(e_s-a_s)(e_s-a_r-1)\geq 0.\]

Case $k+1>r+1$: One can rewrite $\sigma$ as
\[\sigma=(r-1)+a_r+\cdots+a_k-(k-2).\] 
Then, using Eq.~(\ref{pepe-ineq2}), we obtain that
Eq.~(\ref{pepe-ineq1}) reduces to
\[(a_ra_s+a_s-1)e_{r+2}\cdots e_{k+1}\geq\left(\sigma\right)e_{k+1}-1.\]
But, as $a_s\geq 2$, using Lemma~\ref{basic-ineq}, we get
 \begin{eqnarray*}(a_ra_s+a_s-1)e_{r+2}\cdots e_{k}&\geq&
(a_ra_s+a_s-1)+e_{r+2}+\cdots+e_k-(k-r-1)\\
&\geq&(a_r+a_s)+a_{r+2}+\cdots+a_k+r-k+1=\sigma.
 \end{eqnarray*}
\quad So, multiplying by $e_{k+1}$, the
required inequality follows.
\end{proof}

\begin{theorem}\label{pepe-vila} Let $d\geq 1$ and 
$1\leq e_1\leq\cdots\leq e_m$ be integers. Suppose 
$1\leq a_i\leq e_i$ and $1\leq b_i\leq e_i$, for $i=1,\ldots,m$, are integers such that 
$d=\sum_{i=1}^m a_i=\sum_{i=1}^m
b_i$ and $a\neq b$. Then 
\[\pi(a,b)\geq
\left(\sum_{i=1}^m a_i-\sum_{i=k+1}^me_i-(k-2)\right)e_{k+1}\cdots
e_m-e_{k+2}\cdots e_m\] for $k=1,\ldots,m-1$, where $e_{k+2}\cdots e_m=1$ when
$k=m-1$.
\end{theorem}

\begin{proof}  Apply to $(a,b)$ any of the four
``operations'' described in Lemma~\ref{lemma2-pepe}, and let $(a', b')$ be the new obtained
pair. These operations should be applied in such a way that $1\leq
a_i',b_i'\leq e_i$ for $i=1,\ldots,m$ and $a'\neq b'$; this is called
a {\it valid\/} operation. One can order
the set of all pairs $(a,b)$ that satisfy the hypothesis of the proposition
using the GRevLex order defined by $(a,b)\succ(a',b')$ if and only if
the last non-zero entry of $(a,b)-(a',b')$ is negative. Note that by construction
$d=\sum a_i'=\sum b_i'=\sum a_i$. Repeat this step as many times as possible
(which is a finite number because the result $(a',b')$ of any valid
operation applied to $(a,b)$ satisfies $(a,b)\succ(a',b')$). 
Permitting an abuse of notation, let $a$ and $b$ be the
resulting sequences at the end of that process. We will show that these $a$ and
$b$ satisfy the hypothesis of Lemma~\ref{lemma3-pepe}.

Set $r=\min(i:a_i\neq b_i)$. By symmetry we may assume $a_r<b_r$. 
Pick the first $s>r$ such that $a_s>b_s$ (the case $a_r>b_r$
and $a_s<b_s$ can be shown similarly). 

Claim (a): For $p<r$, $a_p=1$. Assume $a_p>1$. If $a_p>a_r$, we can
apply Lemma~\ref{lemma2-pepe}$(i)[p,r]$, which is 
assumed not possible; (this last notation means that we are applying
Lemma~\ref{lemma2-pepe}$(i)$ 
with the indexes $p$ and $r$). Otherwise apply Lemma~\ref{lemma2-pepe}$(ii)[p,r]$. So
$a_p=b_p=1$ for $p<r$.

Claim (b): $s=r+1$. Suppose $r<p<s$. To obtain a contradiction it
suffices to show that we can apply a valid operation to $a,b$.  By
the choice of $s$, $a_p\leq b_p$. 
If $b_r > b_p$, we can apply Lemma~\ref{lemma2-pepe}$(i)[r,p]$. If
$b_p > b_s$, we can apply Lemma~\ref{lemma2-pepe}$(i)[p,s]$. Hence
$b_r\leq b_p \leq b_s$. Notice that $b_p\geq 2$ because 
$a_r<b_r\leq b_p$. If $a_p = b_p$, we can apply
Lemma~\ref{lemma2-pepe}$(ii)[p,s]$. If $a_p<b_p$, we can apply
Lemma~\ref{lemma2-pepe}$(iii)[p,s]$ because $a_p<b_p\leq b_s<a_s$. 

Claim (c): For $p>s$, $a_p=b_p=e_p$. If $b_p<a_p$, applying Claim
(b) to $r$ and $p$ we get a contradiction. Thus we may
assume $b_p\geq a_p$. It suffices to show that $a_p=e_p$. 
If $a_s>a_p$, then by Lemma~\ref{lemma2-pepe}$(i)[s,p]$
one can apply a valid operation to $a,b$, a contradiction. Thus
$a_s\leq a_p$. If $a_p<e_p$, then $b_s<a_s\leq a_p<e_p$, and by
Lemma~\ref{lemma2-pepe}$(iii)[s,p]$ we can apply  
a valid operation to $a,b$, a contradiction. Hence $a_p=e_p$. 

Claim (d): $b_r\leq a_s$ and $b_s=a_s-1$. By the previous claims one
has the equalities $s=r+1$ and $a_r+a_s=b_r+b_s$. If $a_s<b_r$. Then
$b_s<a_s<b_r$, and by Lemma~\ref{lemma2-pepe}$(i)[r,s]$ we can apply a
valid operation to $a,b$, a contradiction. Hence $a_s\geq b_r$. 
Suppose $a_s=b_r$, then $a_r=b_s$. If $a_s-a_r\geq 2$, by 
Lemma~\ref{lemma2-pepe}$(iv)[r,s]$ we can apply a valid operation to
$a,b$, a contradiction. Hence, in this case, $a_s-b_s=a_s-a_r=1$. 
Suppose $a_s>b_r$. If $b_r>b_s$, then $a_s>b_r>b_s$, and we can use 
Lemma~\ref{lemma2-pepe}$(i)[r,s]$ to apply a valid operation to $a,b$,
a contradiction. Hence $b_r\leq b_s$. 
If $a_s-b_s=b_r-a_r\geq 2$, then $a_r<b_r\leq b_s<a_s$, and by 
Lemma~\ref{lemma2-pepe}$(iii)[r,s]$ we can apply a valid operation to
$a,b$, a contradiction. So, in this other case, also $a_s-b_s=1$.
In conclusion, we have that $b_r\leq a_s$ and $b_s=a_s-1$, as claimed.

From Claims (a)--(d), we obtain that $a,b$ satisfy the hypothesis of 
Lemma~\ref{lemma3-pepe}. Hence the required inequality follows from
Lemmas~\ref{lemma2-pepe} and \ref{lemma3-pepe}. 
\end{proof}

For $\alpha:\alpha_1,\ldots,\alpha_n$ and
$\beta:\beta_1,\ldots,\beta_n$ sequences in
$\Z^+$ we define 
\[P(\alpha,\beta)=\prod_{i=1}^n(d_i-\alpha_i)+
\prod_{i=1}^n(d_i-\beta_i)-\prod_{i=1}^n\min\{d_i-\alpha_i,d_i-\beta_i\}.\]
\begin{lemma}\label{pepe-vila-new-1} Let $1\leq d_1\leq\cdots\leq d_n$,
$0\leq \alpha_i,\beta_i\leq d_i-1$ for $i=1,\ldots,n$, $n\geq 2$, be integers such that
$\sum_{i=1}^n\alpha_i=\sum_{i=1}^n\beta_i$ 
and $(\alpha_1,\ldots,\alpha_n)\neq (\beta_1,\ldots,\beta_n)$. Then 
\begin{eqnarray}\label{aug-18-15}
P(\alpha,\beta)\geq\left(\sum_{i=1}^{k+1}(d_i-\alpha_i)-(k-1)-
\sum_{i=k+2}^n\alpha_i\right)d_{k+2}\cdots d_n-d_{k+3}\cdots d_n  
\end{eqnarray}
for $k=0,\ldots,n-2$, where $d_{k+3}\cdots d_n=1$ if $k=n-2$.
\end{lemma}

\begin{proof} Making the substitutions $m=n$, $k=k-1$,
$d_i-\alpha_i=a_i$, $d_i-\beta_i=b_i$, and
$d_i=e_i$, the inequality follows at once from 
Theorem~\ref{pepe-vila}.
\end{proof}

\section{Second generalized Hamming weight}

Let $A_1,\ldots,A_{s-1}$ be subsets of $\mathbb{F}_q$
and let  
$
\mathbb{X}:=[A_1 \times \cdots\times A_{s-1} \times \{1\}] \subset
\mathbb{P}^{s-1}
$
be a projective cartesian set, 
where $d_i=|A_i|$ for all $i=1,\ldots,s-1$ and $2 \leq d_1 \leq \cdots
\leq d_{s-1}$. The Reed--Muller-type code
$C_{\mathbb{X}}(d)$ is called an {\it affine cartesian code\/}
\cite{cartesian-codes}. If $\mathbb{X}^*=A_1 \times \cdots \times
A_{s-1}$, then $C_{\mathbb{X}}(d)=C_{\mathbb{X}^*}(d)$
\cite{cartesian-codes}. 
Assume
$d=\sum_{i=1}^{k}\left(d_i-1\right)+\ell$, where 
$k,\ell$ are integers such that 
$0\leq k\leq s-2$ and $1\leq \ell \leq
d_{k+1}-1$. 
\begin{lemma} \label{lemma1}
We can find two linearly
independent 
polynomials $F$ and $G$ $\in S_{\leq d}$ such that
$$
|V_{\mathbb{X}^*} (F) \cap V_{\mathbb{X}^*} (G)|= \left\{
\begin{array}{lll}
d_1 \cdots d_{s-1}-(d_{k+1}-\ell+1) d_{k+2} \cdots d_{s-1}+d_{k+3}
\cdots d_{s-1} & {\mbox{if}} & k<s-3, \\ 
d_1 \cdots d_{s-1}-(d_{k+1}-\ell+1) d_{k+2} \cdots d_{s-1}+1 & {\mbox{if}} & k=s-3, \\
d_1 \cdots d_{s-1}-d_{s-1}+\ell-1 & {\mbox{if}} & k=s-2.
\end{array} \right.
$$
\end{lemma} 
\demo 
Case (I): $k \leq s-3$. Similarly to \cite{cartesian-codes} we take 
$ A_i=\{\beta_{i,1}, \ldots,\beta_{i,d_i}\}$, for $i=1,\ldots,s-1$. Also, for $i=1,\ldots,k$, let
 \begin{eqnarray*}
& & f_i:=(\beta_{i, 1}-t_i)(\beta_{i,2}-t_i) \cdots (\beta_{i,
d_{i-1}}-t_i),\\
 & & g:=(\beta_{k+1, 1}-t_{k+1})(\beta_{k+1,2}-t_{k+1}) \cdots (\beta_{k+1, \ell-1}-t_{k+1}).
 \end{eqnarray*}
 
 Setting $h_1:=\beta_{k+1, \ell}-t_{k+1}$ and $h_2:=\beta_{k+2,
 \ell}-t_{k+2}$. We define $F:=f_1 \cdots f_k \cdot g \cdot h_1$ and
 $G:=f_1 \cdots f_k \cdot g \cdot h_2$. Notice that $\deg F=\deg
 G=\sum_{i=1}^k (d_i-1)+\ell=d$ and that they are linearly independent over $\mathbb{F}_q$. Let   
 \begin{eqnarray*}
& & V_1:=(A_1 \times \cdots \times A_{s-1}) \setminus
(V_{\mathbb{X}^*} (F) 
\cap V_{\mathbb{X}^*} (G)),\\ 
& & V_2:=\{\beta_{1,d_1}\} \times \cdots\times \{\beta_{k,d_k}\} \times
 \{\beta_{k+1,i}\}_{i=\ell}^{d_{k+1}} \times A_{k+2} \times
 \cdots\times A_{s-1}.
 \end{eqnarray*}
 
 It is easy to see that $V_1 \subset V_2$ and 
$(V_2 \setminus V_1) \cap (V_{\mathbb{X}^*}(F) \cap
V_{\mathbb{X}^*}(G))=V_3$, where
 $$
 V_3=\left\{ \begin{array}{lll}
 \{\beta_{1,d_1}\} \times \cdots\times \{\beta_{k,d_k}\} \times
 \{\beta_{k+1,\ell}\} \times \{\beta_{k+2,\ell}\} \times A_{k+3}
 \times \cdots\times A_{s-1} & {\mbox{if}} & k<s-3, \\
 \{\beta_{1,d_1}\} \times \cdots\times \{\beta_{k,d_k}\} \times
 \{\beta_{k+1,\ell}\} \times \{\beta_{k+2,\ell}\} & {\mbox{if}} & k=s-3. 
 \end{array} \right.
 $$
 
 Therefore
 $$
 |V_1|=|V_2|-|V_3|=\left\{
 \begin{array}{lll}
 (d_{k+1}-\ell+1)d_{k+2} \cdots d_{s-1}-d_{k+3} \cdots d_{s-1} & {\mbox{if}} & k<s-3, \\
 (d_{k+1}-\ell+1) d_{k+2} \cdots d_{s-1}-1 & {\mbox{if}} & k=s-3,
 \end{array} \right.
 $$
 and the claim follows because $|V_{\mathbb{X}^*}(F) \cap
 V_{\mathbb{X}^*}(G)|=d_1 \cdots d_{s-1}-|V_1|$. 
 
Case (II): $k=s-2$. As $\ell \leq d_{k+1}-1$ then $\ell+1 \leq d_{k+1}$. Let
$h_3:=\beta_{k+1,\ell+1}-t_{k+1}$, and $F, f_i, g, h_1$ as in Case (I).
Let $G':=f_1 \cdots f_k \cdot g \cdot h_3$. If  
\begin{eqnarray*}
 & & V_1':=(A_1 \times \cdots \times A_{s-1}) \setminus
 (V_{\mathbb{X}^*} (F) \cap V_{\mathbb{X}^*} (G')), \\
  & &V_2':=\{\beta_{1,d_1}\} \times \cdots\times \{\beta_{k,d_k}\}
  \times \{\beta_{k+1,i}\}_{i=\ell}^{d_{k+1}}, 
 \end{eqnarray*}
 then (because $h_1$ and $h_3$ do not have common zeros) $V_1'=V_2'$ and thus
 $$
 |V_1'|=d_{k+1}-\ell+1=d_{s-1}-\ell+1.
 $$
 
 The result follows because $|V_{\mathbb{X}^*} (F) \cap V_{\mathbb{X}^*} (G')|=d_1 \cdots
 d_{s-1}-|V_1'|$. \qed

\begin{lemma}{\cite[Lemma~3.3]{min-dis-ci}}\label{dec20-15} 
Let $L\subset S$ be the ideal $(t_1^{d_1},\ldots ,t_{s-1}^{d_{s-1}})$,
where $d_1,\ldots,d_{s-1}$ are in $\mathbb{N}_+$. 
If $t^a=t_1^{a_1}\cdots
t_s^{a_s}$, $a_j\geq 1$ for some $1\leq j\leq s-1$, and $a_i\leq
d_i-1$ for $i\leq s-1$,
then 
$$
\deg(S/(L,t^a)) =\deg(S/(L,t_1^{a_1}\cdots
t_{s-1}^{a_{s-1}}))= d_1\cdots d_{s-1}-
\prod_{i=1}^{s-1}(d_i-a_i).
$$
\end{lemma}

We come to one of our applications to coding theory.

\begin{theorem}\label{vila-pepe-sarabia-2} Let $A_i$, $i=1,\ldots,s-1$, be
subsets of $\mathbb{F}_q$ and let $\mathbb{X}\subset\mathbb{P}^{s-1}$ be the 
projective cartesian set given by $\mathbb{X}=[A_1\times\cdots\times
A_{s-1}\times\{1\}]$. If $d_i=|A_i|$ for $i=1,\ldots,s-1$ and
$2\leq d_1\leq \cdots\leq d_{s-1}$, then
$$
\delta_\mathbb{X}(d,2)=\left\{\hspace{-1mm}
\begin{array}{ll}\left(d_{k+1}-\ell+1\right)d_{k+2}\cdots
d_{s-1}-d_{k+3}\cdots d_{s-1}&\mbox{ if }
k<s-3,\\
\left(d_{k+1}-\ell+1\right)d_{k+2}\cdots d_{s-1}-1&\mbox{ if }
k=s-3,\\
\qquad \qquad d_{s-1}-\ell+1&\mbox{ if } k=s-2,
\\
\qquad \qquad  \qquad  2&\mbox{ if } d\geq
\sum\limits_{i=1}^{s-1}\left(d_i-1\right),
\end{array}
\right.
$$
where $0\leq k\leq s-2$ and $\ell$ are integers such that 
$d=\sum_{i=1}^{k}\left(d_i-1\right)+\ell$ and $1\leq \ell \leq
d_{k+1}-1$.
\end{theorem}

\begin{proof} We set $n=s-1$, $I=I(\mathbb{X})$, and
$L=(t_1^{d_1},\ldots,t_n^{d_n})$. By \cite[Theorem~1, Corollary~1]{wei}, 
we get $\delta_\mathbb{X}(d,2)=2$ for
$d\geq \sum_{i=1}^{s-1}(d_i-1)$. Thus we may assume
$d<\sum_{i=1}^{s-1}(d_i-1)$. First we show the inequality
``$\geq$''.  Let $\prec$ be a graded monomial order
with $t_1\succ\cdots\succ 
t_s$. The initial ideal ${\rm in}_\prec(I)$ of $I$ is equal to 
$L=(t_1^{d_1},\ldots,t_n^{d_n})$; see \cite{cartesian-codes}. 
Let $F=\{t^a,t^b\}$ be an element 
of $\mathcal{M}_{\prec,d,2}$, that is, 
$t^a=t_1^{a_1}\cdots t_s^{a_s}$, $t^b=t_1^{b_1}\cdots t_s^{b_s}$,
$d=\sum_{i=1}^sa_i=\sum_{i=1}^sb_i$, $a\neq b$, $a_i\leq d_i-1$ and $b_i\leq d_i-1$
for $i=1,\ldots,n$, and $(L\colon(F))\neq L$. In
particular, from the last condition it follows readily that 
$a_i\neq 0$ and $b_j\neq 0$ for some $1\leq i,j\leq n$.  
There are exact sequences of graded $S$-modules  
\begin{eqnarray*}
&0\rightarrow (S/((L,t^a)\colon t^b))[-|b|]\stackrel{t^b}{\rightarrow}
S/(L,t^a)\rightarrow S/(L,t^a,t^b)\rightarrow 0,\ \ \ \ \ \ \ \
\ \ \ & \\
&0\rightarrow (S/((L,t^b)\colon t^a))[-|a|]\stackrel{t^a}{\rightarrow}
S/(L,t^b)\rightarrow S/(L,t^a,t^b)\rightarrow 0,&
\end{eqnarray*}
where $|a|=\sum_{i=1}^sa_i$. From the equalities
\begin{eqnarray*}
&((L,t^a)\colon t^b)=(L\colon t^b)+(t^a\colon
t^b)=(t_1^{d_1-b_1},\ldots,t_{n}^{d_n-b_n},
\prod_{i=1}^s t_i^{\max\{a_i,b_i\}-b_i}), &\\
&((L,t^b)\colon t^a)=(L\colon t^a)+(t^b\colon
t^a)=(t_1^{d_1-a_1},\ldots,t_{n}^{d_n-a_n},
\prod_{i=1}^s t_i^{\max\{a_i,b_i\}-a_i}), &
\end{eqnarray*}
it follows that either $((L,t^a)\colon t^b)$ or $((L,t^b)\colon t^a)$ is contained in
$(t_1,\ldots,t_n)$. Hence at least one of these ideals 
has height $n$. Therefore, setting 
$$P(a,b)=\prod_{i=1}^n(d_i-a_i)+
\prod_{i=1}^n(d_i-b_i)-\prod_{i=1}^n\min\{d_i-a_i,d_i-b_i\},$$ 
and using Lemma~\ref{dec20-15} it is not hard to see that the
degree of $S/(L,t^a,t^b)$ is 
$$
\deg(S/(L,t^a,t^b))=\prod_{i=1}^nd_i-P(a,b), 
$$
and the second generalized footprint function of $I$ is 
\begin{equation}\label{apr7-17}
{\rm fp}_I(d,2)=\min\left\{P(a,b)\vert\,
\{t^a,t^b\}\in\mathcal{M}_{\prec,d,2} \right\}.
\end{equation}

Making the substitution
$-\ell=\sum_{i=1}^{k}\left(d_i-1\right)-\sum_{i=1}^sa_i$ and using
the fact that ${\rm fp}_{I(\mathbb{X})}(d,r)$ is less than or equal to 
$\delta_\mathbb{X}(d,r)$ (see Theorem~\ref{rth-footprint-lower-bound}) it suffices
to show the inequalities 
\begin{equation}\label{apr7-17-1}
P(a,b)\geq \left(\sum_{i=1}^{k+1}(d_i-a_i)-(k-1)-
a_s-\sum_{i=k+2}^na_i\right)d_{k+2}\cdots d_n-d_{k+3}\cdots d_n,
\end{equation}
for $\{t^a,t^b\}\in\mathcal{M}_{\prec,d,2}$ if $0\leq k\leq s-3$, 
where $d_{k+3}\cdots d_n=1$ if $k=s-3$, and
\begin{equation}\label{apr7-17-2}
P(a,b)\geq \sum_{i=1}^{s-1}(d_i-a_i)-(s-3)-
a_s,
\end{equation}
for $\{t^a,t^b\}\in\mathcal{M}_{\prec,d,2}$ if $k=s-2$. As
$(a_1,\ldots,a_n)$ is not equal to $(b_1,\ldots,b_n)$, one has that either
$\prod_{i=1}^nd_i-\prod_{i=1}^n(d_i-a_i)\geq 1$ or
$\prod_{i=1}^nd_i-\prod_{i=1}^n(d_i-b_i)\geq 1$. If $a_s\geq 1$
or $b_s\geq 1$ (resp. $a_s=b_s=0$), the inequality of
Eq.~(\ref{apr7-17-1}) follows at once from 
\cite[Proposition~5.7]{hilbert-min-dis} (resp.
Lemma~\ref{pepe-vila-new-1}).  If $a_s\geq 1$ or $b_s\geq 1$ (resp.
$a_s=b_s=0$), the inequality of Eq.~(\ref{apr7-17-2}) follows at once from 
\cite[Proposition~5.7]{hilbert-min-dis} (resp.
Lemma~\ref{basic-ineq}). This completes the proof of the 
inequality ``$\geq$''.

The inequality ``$\leq$'' follows directly from Lemma \ref{lemma1}.
\end{proof}

Using Theorem~\ref{vila-pepe-sarabia-2} one recovers the following main result of 
\cite{camps-sarabia-sarmiento-vila}.

\begin{corollary}{\cite[Theorem~18]{camps-sarabia-sarmiento-vila}}\label{theorem3}
Let $K=\mathbb{F}_q$ be a finite field and let $\mathbb{T}$ be a
projective  torus in $\mathbb{P}^{s-1}$. If $d\geq 1$ and $s\geq 3$, then 
\begin{equation*}
\delta_\mathbb{T}(d,2)= 
\left\{
\begin{array}{lll}
(q-1)^{s-(k+3)}[(q-1)(q-\ell)-1] & {\mbox{if}} & 1 \leq d \leq \eta, \\
q-\ell & {\mbox{if}} & \eta < d < \gamma, \\
2 & \mbox{if} & d \geq \gamma,
\end{array} \right. 
\end{equation*}
where $k$ and $\ell$ are the unique integers such that
$d=k(q-2)+\ell$, $k \geq 0$, $1 \leq \ell \leq q-2$,
$\eta=(q-2)(s-2)$ and $\gamma=(q-2)(s-1)$.  
\end{corollary}

Another of our applications to coding theory is the following
purely combinatorial formula for the second generalized Hamming weight
of an affine cartesian code which is quite different from the
corresponding formula
of \cite[Theorem~5.4]{GHWCartesian}.

\begin{theorem}\label{application-to-rmtc-1} Let $\mathcal{P}_d$ be
the set of all pairs $(a,b)$, 
$a,b$ in $\mathbb{N}^s$, $a=(a_1,\ldots,a_s)$, $b=(b_1,\ldots,b_s)$, 
such that $a\neq b$, $d=\sum_{i=1}^sa_i=\sum_{i=1}^sb_i$, $1\leq
a_i,b_i\leq d_i-1$ for $i=1,\ldots,n$, $n:=s-1$, $a_i\neq 0$ and $b_j\neq 0$
for some $1\leq i,j\leq n$. If $\mathbb{X}=[A_1\times\cdots\times
A_{n}\times\{1\}]\subset \mathbb{P}^n$, with
$A_i\subset\mathbb{F}_q$, $d_i=|A_i|$, and 
$2\leq d_1\leq\cdots\leq d_{n}$, 
then
$$
{\rm fp}_{I(\mathbb{X})}(d,2)=\delta_\mathbb{X}(d,2)=\min\left\{P(a,b)\vert\,
(a,b)\in \mathcal{P}_d\right\}\ \mbox{ for }\ 
d\leq\textstyle\sum_{i=1}^n (d_i-1),
$$
where $P(a,b)=\prod_{i=1}^{n}(d_i-a_i)+
\prod_{i=1}^{n}(d_i-b_i)-\prod_{i=1}^{n}\min\{d_i-a_i,d_i-b_i\}$.
\end{theorem}

\begin{proof} Let $\psi(d)$ be the formula for
$\delta_\mathbb{X}(d,2)$ given in Theorem~\ref{vila-pepe-sarabia-2}.
Then using Eqs.~(\ref{apr7-17-1}) and Eqs.~(\ref{apr7-17-2}) one has
$\psi(d)\leq {\rm fp}_{I(\mathbb{X})}(d,2)$. By
Theorem~\ref{rth-footprint-lower-bound} one has 
${\rm fp}_{I(\mathbb{X})}(d,r)\leq \delta_\mathbb{X}(d,r)$, and by 
Lemma \ref{lemma1} one has $\delta_\mathbb{X}(d,r)\leq\psi(d)$.
Therefore
$$
\psi(d)\leq{\rm
fp}_{I(\mathbb{X})}(d,2)\leq\delta_\mathbb{X}(d,r)\leq\psi(d).
$$
Thus we have equality everywhere and  the result follows
from Eq.~(\ref{apr7-17}).
\end{proof}

\begin{remark}
Let $\psi(d)$ be the formula for
$\delta_\mathbb{X}(d,2)$ given in Theorem~\ref{vila-pepe-sarabia-2}.
Then 
$$\psi(d)=\min\left\{P(a,b)\vert\,
(a,b)\in \mathcal{P}_d\right\}$$ 
for $d\leq\textstyle\sum_{i=1}^n (d_i-1)$. This equality is
interesting in its own right. 
\end{remark}

\bibliographystyle{plain}

\begin{thebibliography}{10}

\bibitem{GHWCartesian} P. Beelen and M. Datta, 
Generalized Hamming weights of affine Cartesian codes, 
Finite Fields Appl. {\bf 51} (2018), 130--145. 

\bibitem{BHer}{W. Bruns and J. Herzog, 
{\em Cohen-Macaulay Rings\/},  Revised Edition, Cambridge University 
Press, 1997.}

\bibitem{carvalho} C.  Carvalho, On the second Hamming weight of some
Reed-Muller type codes, Finite Fields Appl. {\bf 24} (2013), 
88--94.

\bibitem{carvalho-lopez-lopez} C. Carvalho, V. G. Lopez Neumann and  H.
H. L\'opez, Projective nested cartesian codes, Bull. Braz. Math. Soc.
(N.S.) {\bf 48} (2017), no. 2, 283--302.

\bibitem{CLO} D. Cox, J. Little and D. O'Shea, {\it Ideals, 
Varieties, and Algorithms\/}, Springer-Verlag, 1992.

\bibitem{ghorpade} M. Datta and S. Ghorpade, Number of solutions of
systems of homogeneous polynomial equations over finite fields, Proc.
Amer. Math. Soc. {\bf 145} (2017), no. 2, 525--541.

\bibitem{duursma-renteria-tapia} I. M. Duursma, C. Renter\'\i a and
H. Tapia-Recillas,  
Reed-Muller codes on complete intersections, Appl. Algebra Engrg.
Comm. Comput.  {\bf 11}  (2001),  no. 6, 455--462.

\bibitem{Eisen}{D. Eisenbud, {\it Commutative Algebra with a view
toward Algebraic Geometry\/}, Graduate
Texts in  Mathematics {\bf 150}, Springer-Verlag, 1995.}

\bibitem{geil} O. Geil, 
On the second weight of generalized Reed-Muller codes, Des. Codes
Cryptogr. {\bf 48} (2008), 323--330. 

\bibitem{geil-2008} O. Geil, Evaluation codes from an affine variety
code perspective, Advances in algebraic geometry codes, 153--180, 
Ser. Coding Theory Cryptol., 5, World Sci. Publ., Hackensack, NJ,
2008.

\bibitem{geil-hoholdt} 
O. Geil and T. H{\o}holdt, 
Footprints or generalized Bezout's theorem, 
IEEE Trans. Inform. Theory {\bf 46}  (2000),  no. 2, 635--641.

\bibitem{geil-pellikaan} O. Geil and R. Pellikaan, 
On the structure of order domains, Finite Fields Appl. {\bf 8}
(2002), no. 3, 369--396. 

\bibitem{geil-thomsen} O. Geil and C. Thomsen, 
Weighted Reed--Muller codes revisited, Des. Codes Cryptogr. {\bf 66}
(2013), 195--220.

\bibitem{camps-sarabia-sarmiento-vila} 
M. Gonz\'alez-Sarabia, E. Camps, E. Sarmiento and R. H. Villarreal, 
The second generalized Hamming weight of some evaluation codes
arising from a 
projective torus, Finite Fields Appl. {\bf 52} (2018), 370--394.

\bibitem{GHW2014} M. Gonz\'alez--Sarabia and C. Renter\'{\i}a,
Generalized Hamming weights and some parameterized codes, Discrete
Math. {\bf 339} (2016), 813--821.  

\bibitem{GRT} M. Gonz\'alez-Sarabia, C. Renter\'\i a and H.
Tapia-Recillas, Reed-Muller-type codes over the Segre variety,  
Finite Fields Appl. {\bf 8}  (2002),  no. 4, 511--518. 

\bibitem{mac2} D. Grayson and M. Stillman, 
{\em Macaulay\/}$2$, 1996. 
Available via anonymous ftp from {\tt math.uiuc.edu}.

\bibitem{Pellikaan} P. Heijnen and R. Pellikaan, Generalized Hamming
weights of $q$--ary Reed--Muller codes, IEEE Trans. Inform. Theory {\bf
44} (1998), no. 1, 181--196.

\bibitem{helleseth} T. Helleseth, T. Kl{\o}ve and J. Mykkelveit, 
The weight distribution of irreducible cyclic codes with block
lengths $n_1((q^l-1)/N)$, Discrete Math. {\bf 18} (1977), 179--211.

\bibitem{Johnsen} T. Johnsen and H. Verdure, Generalized Hamming
weights for almost affine codes, IEEE Trans. Inform. Theory {\bf 63}
(2017), no. 4, 1941--1953.


\bibitem{klove} T. Kl{\o}ve, The weight distribution of linear codes
over $GF(q^l)$ having generator matrix over $GF(q)$, Discrete Math.
{\bf 23} (1978), no. 2, 159--168.

\bibitem{lachaud} G. Lachaud, The parameters of projective
Reed-Muller codes, Discrete Math. {\bf 81} (1990),  no. 2, 217--221.

\bibitem{cartesian-codes} H. H. L\'opez, C. Renter\'\i a and R. H.
Villarreal, Affine cartesian codes,
Des. Codes Cryptogr. {\bf 71} (2014), no. 1, 5--19.

\bibitem{MacWilliams-Sloane} F. J. MacWilliams and N. J. A. Sloane, 
The Theory of Error-correcting Codes, North-Holland, 1977. 

\bibitem{hilbert-min-dis} J. Mart\'\i nez-Bernal, Y. Pitones 
and R. H. Villarreal, Minimum
distance functions of graded ideals   
and Reed-Muller-type codes, 
J. Pure Appl. Algebra {\bf 221} (2017), 251--275. 

\bibitem{min-dis-ci} J. Mart\'\i nez-Bernal, Y. Pitones and R. H. Villarreal, 
Minimum distance functions of complete intersections, J. Algebra
Appl., to appear.

\bibitem{mercier-rolland} D. J. Mercier and R. Rolland, 
Polyn\^omes homog\`enes qui s'annulent sur l'espace projectif
$\mathbb P^m(\mathbb F_q)$, J. Pure Appl. Algebra {\bf 124} (1998),
227--240. 

\bibitem{footprint-ci}  L. N\'u\~nez-Betancourt, Y. Pitones and R. H.
Villarreal, Footprint and minimum distance functions,  
Commun. Korean Math. Soc. {\bf 33} (2018), no. 1, 85--101.


\bibitem{prim-dec-critical}  L. O'Carroll, F. Planas-Vilanova and R. H. Villarreal,
Degree and algebraic properties of lattice and matrix ideals, 
SIAM J. Discrete Math. {\bf 28} (2014), no. 1, 394--427. 

\bibitem{olaya} W. Olaya--Le\'on and C. Granados--Pinz\'on, The
second generalized Hamming weight of certain Castle codes, Des. Codes
Cryptogr. {\bf 76} (2015), no. 1, 81--87.

\bibitem{algcodes} C. Renter\'\i a, A. Simis and R. H. Villarreal,
Algebraic methods for parameterized codes 
and invariants of vanishing ideals over finite fields, Finite Fields
Appl. {\bf 17} (2011), no. 1, 81--104.  

\bibitem{schaathun-willems} H. G. Schaathun and W. Willems, A lower bound on the weight
hierarchies of product codes, Discrete Appl. Math. {\bf 128} (2003),
no. 1, 251--261.

\bibitem{sorensen} A. S{\o}rensen, Projective Reed-Muller codes, 
IEEE Trans. Inform. Theory {\bf 37} (1991), no. 6, 1567--1576.

\bibitem{Sta1}{R. Stanley, Hilbert functions of graded 
algebras, Adv.
Math. {\bf 28} (1978), 57--83.}

\bibitem{tsfasman} M. Tsfasman, S. Vladut and D. Nogin, {\it
Algebraic 
geometric codes{\rm:} basic notions}, Mathematical Surveys and
Monographs {\bf 139}, American Mathematical Society, 
Providence, RI, 2007. 

\bibitem{monalg-rev} R. H. Villarreal, {\it Monomial Algebras, Second Edition\/}, 
Monographs and Research Notes in Mathematics, Chapman and Hall/CRC, 2015.

\bibitem{wei} V. K. Wei, Generalized Hamming 
weights for linear codes, 
IEEE Trans. Inform. Theory {\bf 37}  (1991),  no. 5, 1412--1418. 

\bibitem{wei-yang} V. K. Wei and K. Yang, 
On the generalized Hamming weights of product codes, 
IEEE Trans. Inform. Theory {\bf 39}  (1993),  no. 5, 1709--1713. 

\bibitem{Yang} M. Yang, J. Lin, K. Feng and D. Lin, 
Generalized Hamming weights of irreducible cyclic codes, IEEE Trans.
Inform. Theory {\bf 61} (2015), no. 9, 4905--4913. 
\end{thebibliography}

\end{document}